\documentclass[a4paper,12pt]{amsart}
\usepackage{a4wide,fullpage,setspace}
\usepackage{amsmath,amssymb,mathtools,abraces}
\usepackage[T1]{fontenc}
\usepackage[utf8]{inputenc}
\usepackage{lmodern,microtype,accents}
\usepackage[numbers,sort&compress]{natbib}
\usepackage{tikz,tikz-cd}
\usepackage{hyperref}
\usepackage[shortlabels]{enumitem}
\definecolor{dark-red}{rgb}{0.5,0.15,0.15}
\hypersetup{
	colorlinks   = true, 
	urlcolor     = dark-red, 
	linkcolor    = black, 
	citecolor   = dark-red 
}
\swapnumbers

\title[C-direct categories and branching spaces]{C-direct categories and branching spaces}
\author[P. Gaucher]{Philippe Gaucher}
\address{Universit\'e Paris Cit\'e, CNRS, IRIF, F-75013, Paris, France}
\urladdr{\url{https://www.irif.fr/~gaucher}} 

\newtheorem{thm}{Theorem}[section]
\newtheorem{proposition}[thm]{Proposition}
\newtheorem{lem}[thm]{Lemma}
\newtheorem{cor}[thm]{Corollary}
\theoremstyle{definition}
\newtheorem{definition}[thm]{Definition}
\newtheorem{remark}[thm]{Remark}
\newtheorem{notation}[thm]{Notation}

\makeatletter
\let\leq\@undefined
\let\geq\@undefined
\let\vec\@undefined
\let\phi\@undefined
\let\epsilon\@undefined
\makeatother

\newcommand{\epsilon}{\varepsilon}
\newcommand{\phi}{\varphi}
\newcommand{\vec}{\overrightarrow}
\newcommand{\leq}{\leqslant}
\newcommand{\geq}{\geqslant}
\newcommand{\Set}{\mathrm{Set}}
\newcommand{\PoSet}{\mathrm{PoSet}}
\newcommand{\A}{\mathcal A}
\newcommand{\Sq}{\square}
\newcommand{\SqhatS}{\wh{\Sq}_S}
\newcommand{\id}{\operatorname{id}}
\newcommand{\op}{\mathrm{op}}
\newcommand{\dd}{\overrightarrow d_1}
\newcommand{\T}{\operatorname{T}}
\newcommand{\C}{\mathcal C}
\newcommand{\Top}{\mathbf{Top}}
\newcommand{\Ob}{\operatorname{Ob}}
\newcommand{\overcat}[1]{\overrightarrow{#1}}
\newcommand{\undercat}[1]{\overleftarrow{#1}}
\newcommand{\levelcat}[1]{\overleftrightarrow{#1}}
\newcommand{\wh}{\widehat}
\newcommand{\lev}{\mathrm{lev}}
\newcommand{\cocartesian}{\arrow[lu, phantom, "\ulcorner"{font=\Large}, pos=0]}

\newcommand{\LA}{\mathcal{L}_{\mathcal{A}}}
\newcommand{\sk}{\operatorname{sk}}
\newcommand{\Pminus}{\mathcal P^-}
\newcommand{\geom}{\mathrm{geom}}
\newcommand{\supp}{\operatorname{supp}}

\begin{document}
\begin{abstract}
For a direct category, the Reedy model structure on covariant diagrams coincides with the projective model structure. We extend this comparison to c-direct categories, in which morphisms may preserve degree as well as raise it. Every c-direct category \(\A\) is c-Reedy. If the projective model structure exists on diagrams in a model category \(\mathcal M\) over each degree stratum, then it exists on \(\mathcal M^\A\) and coincides with the c-Reedy model structure. The latching maps give a degreewise criterion for projective cofibrancy. When \(\A\) has cofibrant representables, any two realization functors over a common diagram are naturally weakly equivalent on cofibrant presheaves.

We apply these results to branching spaces. The concrete transverse sets arising in concurrency are freely generated by precubical sets. For every thick category of cubes \(\A\), we construct a colimit-preserving \(\epsilon\)-branching space functor on all \(\A\)-sets. On a free \(\A\)-set, it recovers the branching space of the generating precubical set by a natural homeomorphism. We also characterize the ordinary box category among product-closed thick categories of cubes by colimit preservation of the geometric short-path construction for every \(0<\epsilon<1\).
\end{abstract}

\maketitle
\setcounter{tocdepth}{1}
\tableofcontents
\hypersetup{linkcolor = dark-red}

\section{Introduction}

\subsection*{The categorical results}
For a direct category \(D\) and a model category \(\mathcal M\), the Reedy and projective model structures on the category \(\mathcal M^D\) of covariant diagrams coincide \cite[Sections~5.1--5.2]{MR99h:55031}. We prove an analogue for a small category \(\A\) equipped with an ordinal degree function for which each morphism preserves or raises degree. We call such a category \emph{c-direct}. A noninvertible endomorphism can then occur within a single degree. In a generalized Reedy category in the sense of Berger and Moerdijk, such an endomorphism would have to factor through a lower degree \cite[Definition~1.1]{BergerMoerdijkReedy}. The c-direct degree function precludes such a factorization.

Every c-direct category is c-Reedy in Shulman's sense \cite{ShulmanReedy} (Theorem~\ref{thm:c-Reedy-triangular}). Write \(\A_{=\delta}\) for the full subcategory of objects of degree \(\delta\). If the projective model structure on \(\mathcal M^{\A_{=\delta}}\) exists for every \(\delta\), then the c-Reedy model structure on \(\mathcal M^\A\) exists and is projective (Theorem~\ref{thm:proj-cof-suff-cond}). Indeed, all matching objects are terminal. The latching maps retain the action of the morphisms within each degree. A diagram \(X\) is projectively cofibrant precisely when
\[
L_\delta X\longrightarrow X|_{\A_{=\delta}}
\]
is a projective cofibration in \(\mathcal M^{\A_{=\delta}}\) for every \(\delta\).

A diagram \(X:\A\to\mathcal M\) determines a colimit-preserving functor on presheaves by the coend \(\wh X(K)=\int^{a\in\A}K(a)\cdot X(a)\). Suppose that the representable presheaves are cofibrant with respect to the boundary inclusions \(\partial\A[a]\subseteq\A[a]\). Consider two diagrams over \(I:\A\to\mathcal M\) whose maps to \(I\) are objectwise trivial fibrations and whose associated functors take boundary inclusions to cofibrations. The latching criterion gives a common realization and a natural zigzag of weak equivalences between the two realizations on every cofibrant presheaf (Theorems~\ref{thm:comparison-realization-1} and~\ref{thm:comparison-realization-2}). This generalizes the comparisons obtained for precubical sets in \cite{NaturalRealization} and for symmetric transverse sets in \cite{DirectedDegeneracy}.

\subsection*{The application to branching spaces}
Precubical sets are among the standard combinatorial models of concurrent processes in directed algebraic topology \cite{DAT_book}. In a geometric model of a program, executions are directed paths. Grandis's fundamental category has the points of a directed space as objects. Its morphisms are directed paths modulo the equivalence relation generated by directed homotopies fixing both endpoints \cite{mg,Grandis}. The static analyser ALCOOL uses the category of future components to reduce the state space of a concurrent program \cite[p.~181]{GoubaultHeindelMimram}.

The concrete examples of concurrent processes are the same in the precubical and transverse settings. The transverse sets arising as models in these applications are precisely the free extensions \(\LA(K)\) of precubical sets \(K\). Here \(\A\) is a category of cubes containing the box category \(\Sq\), and \(\LA\) is left Kan extension along \(\Sq^\op\hookrightarrow\A^\op\). The additional cube maps give new operations on these examples. For instance, when \(\A\) contains the map \((x,y)\mapsto(\min(x,y),\max(x,y))\) of the square, precomposition with this map gives an idempotent transverse degeneracy on \(2\)-cubes \cite{DirectedDegeneracy}.

Parallel composition with synchronization in process algebra illustrates the role of these degeneracies. The usual labelled \(1\)-coskeleton on precubical sets can produce several distinct \(n\)-cubes for the same family of \(n\) concurrent actions. The precubical construction therefore selects shells by means of the nonfunctorial labelled directed coskeleton. In the maximal category of cubes \(\SqhatS\), the transverse degeneracies and the symmetries supply the missing cube maps. The usual labelled \(1\)-coskeleton then gives a functorial construction of the synchronized product, and the synchronized product of two free extensions is canonically isomorphic to the free extension of their precubical synchronized product \cite[Theorems~4.1.8 and~4.2.6]{symcub}.

The branching space of a precubical set \(K\) records how an execution can begin at each vertex. In the geometric construction, the summand at a vertex \(v\) consists of natural directed paths of fixed length \(0<\epsilon<1\) starting at \(v\) in cubes with initial vertex \(v\) (Definition~\ref{def:old-branching-space}). This construction preserves colimits and gives rise to a branching homology invariant under the uniform cubical subdivisions \(\operatorname{Sub}_p\) \cite[Theorem~5.5 and Corollary~8.4]{cubical-branching}.

In a directed graph the branching space at a vertex is the discrete set of its outgoing edges, since each edge contributes exactly one natural path of length \(\epsilon\) from its source. Consider two rooted graphs. The first consists of a single \(a\)-labelled edge to a terminal vertex. The second consists of two \(a\)-labelled edges from the root to distinct terminal vertices. Pairing the roots and relating the first terminal vertex to each of the other two gives a strong bisimulation of labelled transition systems in the usual sense \cite{0683.68008}. The branching spaces of the underlying precubical sets are a point and a discrete two-point space, respectively. Thus even \(\pi_0\) of the branching space distinguishes these strongly bisimilar systems.

We seek an extension of the precubical branching space to \(\A\)-sets which preserves colimits and recovers the branching space of \(K\) on its free extension \(\LA(K)\). To see why the geometric construction can fail to preserve colimits, consider the quotient of a square by interchange of its coordinates. Two short paths can follow the same initial segment off the diagonal, meet the diagonal, and then have distinct continuations exchanged by the reflection. They have the same image in the quotient square and belong to distinct orbits of whole paths under coordinate interchange (Lemma~\ref{lem:c-case}). Thus the geometric construction fails to preserve this coequalizer. Among product-closed thick categories of cubes, the ordinary box category is the only one for which it preserves colimits for every \(0<\epsilon<1\) (Theorem~\ref{thm:no-colimit-preserving}).

For every thick category of cubes \(\A\), both \(\A\) and its subcategory \(\A^-\) of maps preserving the initial vertex have cofibrant representables (Theorems~\ref{thm:thick-notcaracterization} and~\ref{thm:Aminus-representablecofibrant}). We define \(B^-_\epsilon\) by a coend over \(\A^-\) using the spaces of short natural paths in the individual cubes (Definition~\ref{def:branching-space}). This functor preserves colimits on all \(\A\)-sets. The path diagram gives a realization over the diagram which is empty in degree zero and a point in each positive degree (Theorem~\ref{thm:Bminus-realization}). On cofibrant \(\A\)-sets any realization over this same diagram gives the same homotopy type. For the short-path diagrams themselves, the rescaling following Corollary~\ref{cor:Bminus-invariance} gives a natural homeomorphism for different choices of \(\epsilon\) on all \(\A\)-sets. For every precubical set \(K\), Theorem~\ref{thm:comparison-with-precubical-branching} gives a natural homeomorphism
\[
B^-_\epsilon(\LA(K))\cong \Pminus_\Sq(K,\epsilon).
\]
Thus adjoining the transverse operators freely recovers the original branching space on all the concrete examples.

\section{c-direct categories}
\label{sec:c-direct}

Let \(\A\) be a small category. An \emph{\(\A\)-set} is a presheaf on \(\A\), that is, a functor
\[
K:\A^\op\longrightarrow\Set.
\]
The category of \(\A\)-sets is denoted by
\[
\A^\op\Set.
\]

\begin{definition} \label{def:c-direct}
	A \emph{c-direct category} is a small category \(\A\) equipped with a degree function
	\[
	d:\Ob(\A)\longrightarrow \mathrm{Ord}
	\]
	with values in the class of ordinals such that, for every morphism \(f:x\to y\) of \(\A\), one has
	\[
	d(x)\leq d(y).
	\] 
\end{definition}

Every direct category is c-direct. This notion should not be confused with Shulman's stratified categories: by \cite[Theorem~4.11]{ShulmanReedy}, a small category is stratified precisely when it admits an ordinal degree function for which every morphism non-strictly decreases degree. Thus, with the same degree function, the opposite of a c-direct category is stratified in Shulman's sense. The comparison relevant for the present paper is instead with Shulman's c-Reedy categories; see Theorem~\ref{thm:c-Reedy-triangular} below.

\begin{notation}
	For the rest of the paper, a c-direct category \(\A\) is fixed.
\end{notation}

Inequalities such as \(n\geq 0\), when \(n\) denotes a degree, are understood to range over ordinals \(n\), not only over natural numbers.

For \(n\geq 0\), let \(\A_{<n}\) denote the full subcategory of \(\A\) spanned by the objects \(x\) such that \(d(x)<n\). For objects \(y,z\) of \(\A\), set \cite[Page~37]{ShulmanReedy}
\[
\partial_n\A(y,z)
=
\int^{x\in \A_{<n}}
\A(x,z)\times \A(y,x).
\]
There is a canonical map
\[
\gamma: \partial_n\A(y,z)\longrightarrow \A(y,z)
\]
induced by composition, namely
\[
[u: x\to z,\,v: y\to x]
\longmapsto
u\circ v.
\]
Here \([-]\) denotes the equivalence class in the coend. 

\begin{proposition} \label{prop:triangularity-coend}
	The following hold.
	\begin{enumerate}
		\item If \(d(y)>d(z)\) or \(n\leq d(y)\), then
		\[
		\partial_n\A(y,z)=\varnothing .
		\]
		\item If \(d(y)\leq d(z)\) and \(d(y)<n\), then the canonical map
		\[
		\gamma: \partial_n\A(y,z)\longrightarrow \A(y,z)
		\]
		is a bijection. Thus, through this canonical identification,
		\[
		\partial_n\A(y,z)\cong \A(y,z).
		\]
	\end{enumerate}
\end{proposition}

\begin{proof}
	Recall first that the coend is the quotient of the disjoint union
	\[
	\coprod_{x\in \Ob(\A_{<n})}
	\A(x,z)\times \A(y,x)
	\]
	by the relations
	\[
	(u\circ a,v)\sim (u,a\circ v)
	\]
	for every morphism \(a: x\to x'\) in \(\A_{<n}\), every \(u: x'\to z\), and every \(v: y\to x\). Assume first that \(d(y)>d(z)\). If a summand
	\[
	\A(x,z)\times \A(y,x)
	\]
	were nonempty, then there would be morphisms
	\[
	y\longrightarrow x\longrightarrow z.
	\]
	By the c-direct condition, these two morphisms imply
	\[
	d(y)\leq d(x)\leq d(z),
	\]
	which contradicts \(d(y)>d(z)\). Hence every summand is empty, and therefore
	\[
	\partial_n\A(y,z)=\varnothing .
	\]
	Assume now that \(n\leq d(y)\). If \(x\) is an object of \(\A_{<n}\),
	then
	\[
	d(x)<n\leq d(y),
	\]
	so \(d(y)>d(x)\). By the c-direct condition,
	\[
	\A(y,x)=\varnothing .
	\]
	Thus every summand
	\[
	\A(x,z)\times \A(y,x)
	\]
	is empty, and again
	\[
	\partial_n\A(y,z)=\varnothing .
	\]
	This proves the first assertion. It remains to prove the second assertion. Assume that
	\[
	d(y)\leq d(z)
	\qquad\text{and}\qquad
	d(y)<n.
	\]
	Then \(y\) is an object of \(\A_{<n}\). Define
	\[
	\sigma: \A(y,z)\longrightarrow \partial_n\A(y,z)
	\]
	by
	\[
	\sigma(f)=[f: y\to z,\operatorname{id}_y: y\to y].
	\]
	This is well-defined because \(y\in \A_{<n}\). By construction,
	\[
	\gamma\sigma(f)
	=
	\gamma([f,\operatorname{id}_y])
	=
	f\circ \operatorname{id}_y
	=
	f.
	\]
	Hence
	\[
	\gamma\sigma=\operatorname{id}_{\A(y,z)}.
	\]	
	Conversely, let
	\[
	[u: x\to z,\,v: y\to x]
	\]
	be an element of \(\partial_n\A(y,z)\), with \(x\in \A_{<n}\). Since also \(y\in \A_{<n}\) and \(\A_{<n}\) is a full subcategory of \(\A\), the morphism \(v: y\to x\) belongs to \(\A_{<n}\). Applying the defining coend relation to the morphism \(v: y\to x\) gives
	\[
	[u,v]=[u\circ v,\operatorname{id}_y].
	\]
	Therefore
	\[
	[u,v]
	=
	\sigma(u\circ v)
	=
	\sigma\gamma([u,v]).
	\]
	Thus
	\[
	\sigma\gamma=\operatorname{id}_{\partial_n\A(y,z)}.
	\]
	The maps \(\gamma\) and \(\sigma\) are inverse bijections. Hence
	\[
	\partial_n\A(y,z)\cong \A(y,z)
	\]
	canonically, through the composition map \(\gamma\).
\end{proof}

\begin{proposition} \label{prop:lower-degree-part}
	Fix an object \(x\) of \(\A\). Consider the presheaf
	\[
	F_x
	:=
	\int^{y\in \A_{<d(x)}} \A(y,x)\cdot \A(-,y),
	\]
	where \(\A(-,y)\) is the representable presheaf and where \(\cdot\) denotes the copower of a presheaf by a set. For every object \(z\) of \(\A\), there is a canonical bijection
	\[
	F_x(z)
	\cong
	\begin{cases}
		\A(z,x), & \text{if } d(z)<d(x),\\[4pt]
		\varnothing, & \text{if } d(z)\geq d(x).
	\end{cases}
	\]
	Thus \(F_x\) is the strict lower-degree part of the representable presheaf \(\A(-,x)\). It is denoted by \(\partial\A(-,x)\). 
\end{proposition}

\begin{proof}
	Since colimits in the presheaf category are computed pointwise, for every object \(z\) of \(\A\), we have
	\[
	F_x(z)
	=
	\int^{y\in \A_{<d(x)}} \A(y,x)\times \A(z,y).
	\]
	This coend is the quotient of the disjoint union
	\[
	\coprod_{y\in \A_{<d(x)}} \A(y,x)\times \A(z,y)
	\]
	by the usual coend relations. Explicitly, if
	\[
	a: y\longrightarrow y'
	\]
	is a morphism in \(\A_{<d(x)}\), and if
	\[
	u: y'\longrightarrow x,
	\qquad
	v: z\longrightarrow y,
	\]
	then
	\[
	(u\circ a,v)\sim (u,a\circ v).
	\]
	There is a canonical map
	\[
	\gamma_z: F_x(z)\longrightarrow \A(z,x)
	\]
	induced by composition:
	\[
	[u: y\to x,\ v: z\to y]
	\longmapsto
	u\circ v.
	\]
	Suppose first that \(d(z)\geq d(x)\). For every object \(y\in \A_{<d(x)}\), one has
	\[
	d(y)<d(x)\leq d(z),
	\]
	hence \(d(z)>d(y)\). By the c-direct condition,
	\[
	\A(z,y)=\varnothing.
	\]
	Therefore every summand
	\[
	\A(y,x)\times \A(z,y)
	\]
	is empty, and consequently
	\[
	F_x(z)=\varnothing.
	\]
	Suppose now that \(d(z)<d(x)\). Then \(z\) belongs to the full subcategory \(\A_{<d(x)}\). Hence every morphism
	\[
	f: z\longrightarrow x
	\]
	defines an element of \(F_x(z)\) by the factorization through \(z\):
	\[
	z \xrightarrow{\operatorname{id}_z} z \xrightarrow{f} x.
	\]
	This gives a map
	\[
	\sigma_z: \A(z,x)\longrightarrow F_x(z),
	\qquad
	f\longmapsto [f,\operatorname{id}_z].
	\]
	It is immediate that
	\[
	\gamma_z\sigma_z(f)=f
	\]
	for every \(f: z\to x\).
	
	Conversely, let
	\[
	[u,v]\in F_x(z)
	\]
	be represented by morphisms
	\[
	z\xrightarrow{v} y\xrightarrow{u} x
	\]
	with \(y\in \A_{<d(x)}\). Since \(d(z)<d(x)\), the object \(z\) also belongs to \(\A_{<d(x)}\). Because \(\A_{<d(x)}\) is full, the morphism
	\[
	v: z\longrightarrow y
	\]
	is a morphism of \(\A_{<d(x)}\). Therefore the coend relation applied to \(v\) gives
	\[
	[u,v]=[u\circ v,\operatorname{id}_z].
	\]
	Equivalently,
	\[
	[u,v]=\sigma_z\gamma_z([u,v]).
	\]
	Thus \(\sigma_z\gamma_z=\operatorname{id}_{F_x(z)}\), while we already have \(\gamma_z\sigma_z=\operatorname{id}_{\A(z,x)}\). Hence \(\gamma_z\) is a bijection
	\[
	F_x(z)\cong \A(z,x)
	\]
	when \(d(z)<d(x)\). Combining the two cases proves the assertion.
\end{proof}

\begin{notation}
	For \(x\in \Ob(\A)\), set 
	\[
	\A[x] = \A(-,x), \qquad \partial\A[x] = \partial\A(-,x).
	\]
\end{notation}

\section{c-Reedy categories}
\label{sec:c-reedy}

Let \(f\) be a morphism of \(\A\). The \textit{category of factorizations} of \(f\) has for objects the pairs \[(h,g)\qquad \hbox{such that} \qquad hg=f\] and for morphisms \(k:(h,g)\to (h',g')\) the morphisms \(k\) of \(\A\) (which are called \textit{connecting morphisms}) such that there is a commutative diagram 
\[
\begin{tikzcd}[row sep=3em,column sep=7em]
	\bullet \arrow[r,"g'"] \arrow[equal,d]  & \bullet \arrow[r,"h'"]& \bullet \arrow[equal,d] \\
	\bullet \arrow[r,"g"] & \arrow[u,"k"]\bullet \arrow[r,"h"] & \bullet 
\end{tikzcd}
\]
We have (\cite[Definition~6.12]{ShulmanReedy})
	\begin{enumerate}
		\item A morphism is \textit{level} if its domain and codomain have the same degree.
		\item The \textit{degree} of a factorization \((h,g)\) of a morphism \(f\) is the degree of the intermediate object (i.e. the domain of \(h\) which is the codomain of \(g\)).
		\item A factorization of a morphism \(f\) is \textit{fundamental} if its degree is strictly less than the degrees of both the domain and codomain of \(f\).
		\item A morphism is \textit{basic} if it does not admit any fundamental factorization. 
	\end{enumerate}
Denote by \(\A_{\lev}\) the wide subcategory of \(\A\) whose morphisms are the level morphisms, that is,
\[
\A_{\lev}(x,y)
=
\begin{cases}
	\A(x,y), & d(x)=d(y),\\
	\varnothing, & d(x)\neq d(y).
\end{cases}
\]
This is indeed a subcategory of \(\A\): identities are level, and the composite of two level morphisms is level. The \textit{\(\delta\)-th stratum} of \(\A\), denoted by \[\A_{=\delta},\] is the subcategory of \(\A\) generated by the objects of degree \(\delta\) and by the basic morphisms between them \cite[page~37]{ShulmanReedy}. Since \(\A\) is c-direct, all morphisms are therefore basic. Thus the \(\delta\)-th stratum is the full subcategory generated by all objects of \(\A\) of degree \(\delta\). 

We use the following form of the c-Reedy axioms (see \cite[Definition~8.25]{ShulmanReedy}). A small category \(\C\) equipped with wide subcategories
\(
\overcat{\C},
\undercat{\C},
\levelcat{\C}
\)
called respectively the direct, the inverse and the level subcategories and with a degree function is \textit{c-Reedy} if:
\begin{enumerate}
	\item \(\levelcat{\C}=\overcat{\C}\cap \undercat{\C}\), and its morphisms preserve degree;
	\item every morphism of \(\overcat{\C}\setminus \levelcat{\C}\) strictly raises degree;
	\item every morphism of \(\undercat{\C}\setminus \levelcat{\C}\) strictly lowers degree;
	\item every morphism \(f\) of \(\C\) admits a factorization
	\[
	f=\overcat f\circ \undercat f
	\]
	with \(\undercat f\in \undercat{\C}\) and \(\overcat f\in \overcat{\C}\), and the category of such factorizations is non-empty and connected, with connecting morphisms in \(\levelcat{\C}\);
	\item for every object \(x\), for each \(\delta<d(x)\), the restriction of \(\undercat{\C}(x,-)\) to the degree-\(\delta\) subcategory of \(\levelcat{\C}\) is a coproduct of retracts of representables.
\end{enumerate}

\begin{thm} \label{thm:c-Reedy-triangular}
	Consider the wide subcategories
	\[
	\overcat{\A}=\A,
	\qquad
	\undercat{\A}=\levelcat{\A}=\A_{\lev}.
	\]
	Then \(\A\) is a c-Reedy category for this choice of the direct, inverse, and level subcategories.
\end{thm}

\begin{proof}
	We check the c-Reedy axioms. First, the intersection condition is immediate. Since \(\overcat{\A}=\A\) and \(\undercat{\A}=\A_{\lev}\), one has
	\[
	\overcat{\A}\cap \undercat{\A}
	=
	\A\cap \A_{\lev}
	=
	\A_{\lev}
	=
	\levelcat{\A}.
	\]	
	Next, let \(f:x\to y\) be a morphism of \(\overcat{\A}\setminus \levelcat{\A}\). Since \(\overcat{\A}=\A\), this just means that \(f\) is a morphism of \(\A\) which is not level. The c-direct condition gives
	\[
	d(x)\leq d(y).
	\]
	Since \(f\) is not level, \(d(x)\neq d(y)\). Therefore
	\[
	d(x)<d(y).
	\]
	Thus every non-level morphism of \(\overcat{\A}\) strictly raises degree. The corresponding condition for the inverse part is vacuous, because
	\[
	\undercat{\A}=\levelcat{\A}.
	\]
	Hence there is no morphism in \(\undercat{\A}\setminus\levelcat{\A}\). It remains to verify the factorization axiom. Let
	\[
	f:x\longrightarrow y
	\]
	be a morphism of \(\A\). There is a canonical factorization
	\[
	x \xlongrightarrow{\id_x} x \xlongrightarrow{f} y,
	\]
	where \(\id_x\) belongs to \(\undercat{\A}=\A_{\lev}\) and \(f\) belongs to \(\overcat{\A}=\A\). Thus the category of c-Reedy factorizations of \(f\) is nonempty. Let
	\[
	x \xlongrightarrow{\ell} z \xlongrightarrow{g} y
	\]
	be any other factorization of \(f\), with
	\[
	\ell\in \undercat{\A}=\A_{\lev},
	\qquad
	g\in \overcat{\A}=\A,
	\qquad
	g\ell=f.
	\]
	Since \(\ell\) is level, it belongs to \(\levelcat{\A}\). Therefore \(\ell:x\to z\) defines a morphism in the factorization category from the canonical factorization
	\[
	x \xlongrightarrow{\id_x} x \xlongrightarrow{f} y
	\]
	to the factorization
	\[
	x \xlongrightarrow{\ell} z \xlongrightarrow{g} y.
	\]
	Indeed, the required equalities are
	\[
	\ell\id_x=\ell
	\qquad
	\text{and}
	\qquad
	g\ell=f.
	\]
	Thus \(\ell\) gives a morphism in the factorization category. Hence every factorization is connected to the canonical one. The factorization category is therefore connected. Finally, the remaining c-Reedy axiom concerning the inverse part in strictly smaller degree is automatic. For every object \(x\) and every degree \(\alpha<d(x)\), the restriction of
	\[
	\undercat{\A}(x,-)
	\]
	to the full subcategory of objects of degree \(\alpha\) is the empty functor, because \(\undercat{\A}=\A_{\lev}\) contains only morphisms preserving degree. The empty functor is the empty coproduct of representables. Thus the required representability condition is satisfied. All c-Reedy axioms are therefore satisfied.
\end{proof}

\begin{notation}
	The category of functors from a small category \(\C\) to a locally small category \(\mathcal{M}\) together with the natural transformations is denoted by \(\mathcal{M}^\C\).
\end{notation}

When \(\mathcal M\) is cocomplete, restriction along the Yoneda embedding gives an equivalence between the category of colimit-preserving functors \(\C^\op\Set\to\mathcal M\) and \(\mathcal M^\C\) \cite[Remark~3.2.3]{coend-calculus}. A quasi-inverse sends a diagram \(A\) to the functor \[\wh A(K)=\int^{x\in\C} K(x)\cdot A(x).\]

\begin{notation}
	For the rest of the paper, \(\mathcal{M}\) denotes a model category such that the projective model structure on \[\mathcal{M}^{\A_{=\delta}}\] exists for all \(\delta\geq 0\) (e.g. any combinatorial or accessible model category). 
\end{notation}

\begin{notation} \label{notation:matching-latching-object}
	Let \(n\geq 0\) and \(y\in\Ob(\A_{=n})\). The latching and matching object functors \(L_n,M_n:\mathcal{M}^{\A}\to \mathcal{M}^{\A_{=n}}\) are given by (see \cite[page~37]{ShulmanReedy})
	\begin{align*}
		& (M_nA)_{y} = \int_{z\in \A} A(z)^{\partial_{n}\A(y,z)} \\
		& (L_nA)_{y} = \int^{z\in \A}\partial_n\A(z,y) \cdot A(z) 
	\end{align*}
\end{notation}

We obtain: 

\begin{thm} \label{thm:c-Reedy-model}
	Let \(\mathcal{M}\) be a model category. Suppose that the projective model structure on \(\mathcal{M}^{\A_{=n}}\) exists for all \(n\geq 0\). There exists a unique model structure on \(\mathcal{M}^{\A}\) such that 
	\begin{itemize}
		\item The weak equivalences are objectwise.
		\item A map \(A\to B\) of \(\mathcal{M}^{\A}\) is a fibration (trivial fibration resp.) if and only if, for all \(n\geq 0\), the map \[A(x)\longrightarrow (M_nA)_{x}\times_{(M_nB)_{x}} B(x)\] is a fibration (trivial fibration resp.) of \(\mathcal{M}\) for all \(x\in\Ob(\A_{=n})\).
		\item A map \(A\to B\) of \(\mathcal{M}^{\A}\) is a cofibration (trivial cofibration resp.) if and only if, for all \(n\geq 0\), \[L_nB\sqcup_{L_nA} A|_{\A_{=n}}\longrightarrow B|_{\A_{=n}}\] is a projective cofibration (trivial cofibration resp.) of the projective model structure of \(\mathcal{M}^{\A_{=n}}\).
	\end{itemize}
	This model structure is called the \textit{c-Reedy model structure} of \(\mathcal{M}^{\A}\). 
\end{thm}

\begin{proof}
	By Theorem~\ref{thm:c-Reedy-triangular} and \cite[Theorem~8.26]{ShulmanReedy}, the small category \(\A\) is almost c-Reedy in the sense of \cite[Definition~8.8]{ShulmanReedy}. The proof is complete by \cite[Theorem~8.9]{ShulmanReedy}.
\end{proof}

\begin{lem}[{\cite[Lemma~3.14]{DirectedDegeneracy}}] \label{lem:wlim} 
	Let \(\C\) be a small category. Consider a small diagram \(X:\C\to \mathcal{M}\) and the empty weight \(W:\C\to \Set\) with \(W(c)=\varnothing\) for all \(c\in \C\). Then there is the isomorphism 
	\[
	\int_{c\in \C}X(c)^{W(c)} \cong \mathbf{1}.
	\]
\end{lem}

\begin{lem}[{\cite[Lemma~3.15]{DirectedDegeneracy}}] \label{lem:wcolim} 
	Let \(\C\) be a small category. Consider a small diagram \(X:\C\to \mathcal{M}\) and a weight \(U:\C^{\op}\to \Set\). Let \(\mathcal D\) be the full subcategory of \(\C\) generated by the objects \(c\) such that \(U(c)\neq \varnothing\). Then there is the isomorphism  
	\[
	\int^{c\in \mathcal D} U(c)\cdot X(c) \cong \int^{c\in \C} U(c)\cdot X(c).
	\]
\end{lem}

 Recall that the \textit{projective model structure} is the unique model structure (if it exists) on a functor category \(\mathcal{M}^\C\) such that the weak equivalences and the fibrations are the objectwise ones.

\begin{thm} \label{thm:proj-cof-suff-cond}
Suppose that \(\mathcal{M}^{\A_{=\delta}}\) admits the projective model structure for every degree \(\delta\). The projective model structure on \(\mathcal{M}^{\A}\) exists and coincides with the c-Reedy model structure. Let \(A\in \Ob(\mathcal{M}^{\A})\). Then \(A\) is projectively cofibrant if and only if the map \[\wh A(\partial\A[-]) \longrightarrow \wh A(\A[-])\] is a projective cofibration of \(\mathcal{M}^{\A_{=n}}\) for every \(n\geq 0\).
\end{thm}

\begin{proof}
We mimic the proof of \cite[Theorem~5.17]{DirectedDegeneracy}. The matching object functor \[M_n:\mathcal{M}^{\A}\to \mathcal{M}^{\A_{=n}}\] for all \(n\geq 0\) can be calculated as follows. There is the sequence of isomorphisms of \(\mathcal{M}\)
\[
(M_nA)_{y}\cong \int_{z\in \A} A(z)^{\partial_n\A(y,z)} \cong \int_{z\in \A} A(z)^{\varnothing} \cong \mathbf{1},
\] 
the first isomorphism by definition of the matching object functor (Notation~\ref{notation:matching-latching-object}), the second isomorphism since \(\partial_n\A(y,z)=\varnothing\) by Proposition~\ref{prop:triangularity-coend}, and the third isomorphism by Lemma~\ref{lem:wlim}. Thus, the c-Reedy model structure of Theorem~\ref{thm:c-Reedy-model} on \(\mathcal{M}^{\A}\) coincides with the projective model structure which therefore exists. There is the sequence of isomorphisms of \(\mathcal{M}\)
\[
(L_nA)_{y} \cong \int^{z\in \A}\partial_n\A(z,y) \cdot A(z) \cong \int^{z\in \A_{<n}}\A(z,y) \cdot A(z) \cong \widehat{A}(\partial\A[y]),
\]
the first isomorphism by definition of the latching object functor (Notation~\ref{notation:matching-latching-object}), the second isomorphism by Lemma~\ref{lem:wcolim} and since \(\partial_n\A(z,y)=\varnothing\) for \(d(z)\geq n\) by Proposition~\ref{prop:triangularity-coend}, and finally the third isomorphism by Proposition~\ref{prop:lower-degree-part} and since \(\widehat{A}\) is colimit-preserving. By Theorem~\ref{thm:c-Reedy-model}, \(A\) is projectively cofibrant if and only if for all \(n\geq 0\), the map \(L_nA\to A|_{\A_{=n}}\) is a projective cofibration of the projective model structure of \(\mathcal{M}^{\A_{=n}}\). Since \(A(x) = \widehat{A}(\A[x])\) by definition of \(\widehat{A}\), the proof is complete. 
\end{proof}

\section{Realizations in a model category}
\label{sec:realization}

\begin{definition} \label{def:realization}
	Let \(I\) be an object of \(\mathcal M^\A\). Let \(A\) be an object of \(\mathcal{M}^\A\). The functor \[\wh A:\A^\op\Set\longrightarrow \mathcal M\] is a \textit{realization functor over \(I\)} if the following conditions hold:
	\begin{enumerate}
		\item For all \(x\in\Ob\A\), the map \(\wh A(\partial\A[x])\to \wh A(\A[x])\) is a cofibration of \(\mathcal M\). 
		\item There is a natural transformation 
		\[
		\begin{tikzcd}[cramped]
			A\arrow[twoheadrightarrow,r,"\simeq"] & I
		\end{tikzcd}
		\]
		which is an objectwise trivial fibration of \(\mathcal M\) (i.e. a projective trivial fibration).
	\end{enumerate}
\end{definition}

\begin{definition}
	An object \(K\) of \(\A^\op\Set\) is \textit{cellular} if the canonical map \(\varnothing \to K\) is a  \textit{cellular map}, i.e. a transfinite composition of pushouts of the generating cofibrations \[\partial\A(-,x)\subseteq \A(-,x)\] for \(x\) running over \(\Ob(\A)\). An object \(K\) of \(\A^\op\Set\) is \textit{cofibrant} if it is a retract of a cellular object.
\end{definition}

Let \(x\) be an object of \(\A\). We claim that \(\partial\A(-,x)\) is cofibrant if and only if \(\A(-,x)\) is cofibrant. The forward implication follows immediately from the definition, since \(\A(-,x)\) is obtained from \(\partial\A(-,x)\) by attaching one \(x\)-cell along the generating cofibration
\[
\partial\A(-,x)\subseteq \A(-,x).
\]
Conversely, for every ordinal \(\lambda\), let \(\sk_{<\lambda}\) denote the truncation functor on presheaves defined by
\[
\bigl(\sk_{<\lambda}K\bigr)(y)=
\begin{cases}
	K(y) & \text{if } d(y)<\lambda,\\
	\varnothing & \text{if } d(y)\geq \lambda .
\end{cases}
\]
This is indeed a subpresheaf because \(\A\) is c-direct. Moreover \(\sk_{<\lambda}\) preserves cofibrant (and even cellular) presheaves: it preserves colimits and retracts, and it sends every generating cofibration
\[
\partial\A(-,z)\subseteq \A(-,z)
\]
either to itself, when \(d(z)<\lambda\), or to an identity map, when \(d(z)\geq\lambda\). Taking \(\lambda=d(x)\), one has
\[
\sk_{<d(x)}\A(-,x)
=
\partial\A(-,x).
\]
Therefore, if \(\A(-,x)\) is cofibrant, then \(\partial\A(-,x)\) is cofibrant as well.

Here is an example showing that the representables need not be cofibrant for an arbitrary c-direct category. Let \(G=C_{2}\), and let \(\A\) be the small category with two objects \(0\) and \(1\), with \(d(0)=0\) and \(d(1)=1\), defined by
\[
\A(0,0)=G,\qquad
\A(1,1)=\{\id_{1}\},\qquad
\A(0,1)=\{u\},\qquad
\A(1,0)=\varnothing ,
\]
where composition on \(\A(0,0)\) is the group law and
\[
u\circ g=u
\qquad
(g\in G).
\]
Then \(\A\) is c-direct. We claim that the representable \(\A(-,1)\) is not cofibrant. Indeed, if \(K\) is cofibrant, then \(K(0)\) is a free right \(G\)-set. This follows because attaching a \(0\)-cell adds a copy of the regular right \(G\)-set \(\A(0,0)=G\), whereas attaching a \(1\)-cell does not change the value at \(0\), since
\[
\partial\A(0,1)=\A(0,1)=\{u\}.
\]
Thus, by transfinite induction, the value at \(0\) of any cellular presheaf is a free right \(G\)-set, and the same remains true for retracts, since an equivariant retract of a free \(G\)-set has trivial stabilizers. However
\[
\A(-,1)(0)=\A(0,1)=\{u\}
\]
is the singleton right \(G\)-set with the trivial action \(u\cdot g=u\). Since \(G=C_{2}\) is nontrivial, this right \(G\)-set is not free. Hence \(\A(-,1)\) is not cofibrant.

\begin{definition}
	The c-direct category \(\A\) has \textit{cofibrant representables} if for all \(x\in \Ob(\A)\), the representable \(\A(-,x)\) is cofibrant.
\end{definition}

\begin{thm} \label{thm:comparison-realization-1}
	Assume that \(\A\) has cofibrant representables. Let \[\wh A \quad\hbox{and}\quad\wh B\] be two realization functors over \(I\) such that there exists a commutative diagram of \(\mathcal M^\A\) 
	\[
	\begin{tikzcd}[column sep=7em,row sep=3em]
		A \arrow[r,"\mu"] \arrow[d,twoheadrightarrow,"\simeq"]& B \arrow[d,twoheadrightarrow,"\simeq"]\\
		I \arrow[r,equal] & I
	\end{tikzcd}
	\]
	Then, for every \(\A\)-set \(K\), the morphism \(A\to B\) gives rise to a natural map of \(\mathcal{M}\)
	\[
	\wh\mu_K:\wh{A}(K) \longrightarrow \wh{B}(K)
	\]
	which is a weak equivalence of \(\mathcal{M}\) between cofibrant objects whenever \(K\) is cofibrant.
\end{thm}

\begin{proof}
	First of all, note that for all cofibrant \(\A\)-sets \(K\), \(\wh A(K)\) and \(\wh B(K)\) are cofibrant in \(\mathcal M\). Assume that \(K\) is cellular and fix a cellular presentation of \(K\). For each ordinal \(\lambda\), let \(\C_\lambda(K)\) index its cells of degree \(\lambda\), and write \(|x|\) for the object of \(\A_{=\lambda}\) corresponding to \(x\). The boundary \(\partial\A(-,|x|)\) is supported in degrees strictly less than \(\lambda\), so its attaching map factors through \(\sk_{<\lambda}(K)\). On the stratum \(\A_{=\lambda}\), attaching this cell adds a disjoint copy of \(\A_{=\lambda}(-,|x|)\), and cells of other degrees do not change this stratum. Thus the truncations form a filtration in which all cells of degree \(\lambda\) are attached simultaneously by the following pushout of \(\A\)-sets  
	\[
	\begin{tikzcd}[column sep=3em,row sep=3em]
		\displaystyle\coprod_{x\in \C_\lambda(K)}\partial\A(-,|x|) \arrow[r] \arrow[d] & \sk_{<\lambda}(K) \arrow[d] \\
		\displaystyle\coprod_{x\in \C_\lambda(K)}\A(-,|x|) \arrow[r] & \cocartesian \sk_{<\lambda+1}(K)
	\end{tikzcd}.
	\]
	The comparison map 
	\[
	\varinjlim_{\mu<\lambda} \sk_{<\mu}(K) \xlongrightarrow{\cong} \sk_{<\lambda}(K)
	\]
	is an isomorphism for every limit ordinal \(\lambda\), and the comparison map
	\[
	\varinjlim_{\lambda} \sk_{<\lambda}(K) \xlongrightarrow{\cong} K
	\]
	is an isomorphism. We obtain the commutative cube of \(\mathcal M\)
	\[
	\begin{tikzcd}[column sep=1.5em,row sep=2em]
		\displaystyle\coprod_{x\in \C_\lambda(K)}\wh A(\partial\A[|x|]) \arrow[dd,rightarrowtail]\arrow[rd,"\simeq"]\arrow[rr] && \wh A(\sk_{<\lambda}(K)) \arrow[dd,rightarrowtail] \arrow[rd,"\simeq"]\\
		& \displaystyle\coprod_{x\in \C_\lambda(K)}\wh B(\partial\A[|x|]) \arrow[rr,crossing over] & & \wh B(\sk_{<\lambda}(K)) \arrow[dd,rightarrowtail]\\
		\displaystyle\coprod_{x\in \C_\lambda(K)}\wh A(\A[|x|]) \arrow[rd,"\simeq"]\arrow[rr] && \cocartesian\wh A(\sk_{<\lambda+1}(K)) \arrow[rd,dashed,shorten >=0.6em] \\
		& \displaystyle\coprod_{x\in \C_\lambda(K)}\wh B(\A[|x|]) \arrow[rr] \arrow[uu,leftarrowtail,crossing over] && \cocartesian\wh B(\sk_{<\lambda+1}(K))
	\end{tikzcd}
	\]
	We claim that the map  
	\[
	\wh A(\sk_{<\lambda}(L)) \longrightarrow \wh B(\sk_{<\lambda}(L))
	\]
	is a weak equivalence for every cofibrant \(\A\)-set \(L\). We proceed by transfinite induction on \(\lambda\). At each stage it suffices to prove the claim for cellular \(K\), since the comparison for a retract of \(K\) is a retract of the comparison for \(K\). For \(\lambda=0\), there is nothing to prove. For the successor step, assume the claim at \(\lambda\). For all \(x\in \C_\lambda(K)\), the commutative diagram 
	\[
	\begin{tikzcd}[column sep=7em,row sep=3em]
		A(|x|) \arrow[r,"\mu"] \arrow[d,twoheadrightarrow,"\simeq"]& B(|x|) \arrow[d,twoheadrightarrow,"\simeq"]\\
		I(|x|) \arrow[r,equal] & I(|x|)
	\end{tikzcd}
	\]
	implies the weak equivalence \[\wh A(\A(-,|x|))=A(|x|) \simeq B(|x|) = \wh B(\A(-,|x|))\] by the two-out-of-three property. By definition of \(\partial\A(-,|x|)\), one has 
	\[
	\sk_{<\lambda}(\partial\A(-,|x|)) = \partial\A(-,|x|)\quad\hbox{for all}\quad x\in \C_\lambda(K).
	\]
	These boundaries are cofibrant because the representables are cofibrant. Thus the induction hypothesis implies the weak equivalence \[\wh A(\partial\A(-,|x|)) \simeq \wh B(\partial\A(-,|x|))\] for \(x\in \C_\lambda(K)\). We deduce the same fact for the ordinal \(\lambda+1\) using the commutative cube above and using the cube lemma (\cite[Proposition~15.10.10]{ref_model2} or \cite[Lemma~5.2.6]{MR99h:55031}). It remains to consider the case where \(\lambda\) is a limit ordinal. We have two towers of cofibrations between cofibrant objects of \(\mathcal M\) which are objectwise weakly equivalent. The colimit coincides with the homotopy colimit since the two towers are Reedy cofibrant for the Reedy model structure of \(\mathcal M^\lambda\). This proves the limit ordinal case and completes the proof.
\end{proof}

Recall that the \textit{injective model structure} is the unique model structure (if it exists) on a functor category \(\mathcal{M}^\C\) such that the weak equivalences and the cofibrations are the objectwise ones.

\begin{lem}[{\cite[Lemma~6.10]{DirectedDegeneracy}}] \label{lem:proj-inj}
    Let \(\C\) be a small category and \(\mathcal M\) a model category such that the projective model structure on \(\mathcal M^\C\) exists. Every projective cofibration is an objectwise cofibration. If the injective model structure also exists, the identity of \(\mathcal M^\C\) yields a left Quillen functor
    \[
    (\mathcal M^\C)_{proj}\longrightarrow(\mathcal M^\C)_{inj}.
    \]
\end{lem}

\begin{proof}
    For \(x\in\C\), evaluation at \(x\) has right adjoint \(Y\mapsto(y\mapsto\prod_{\C(y,x)}Y)\). Products preserve fibrations and trivial fibrations, so this right adjoint is right Quillen for the projective model structure. Thus evaluation at \(x\) preserves cofibrations. The last assertion follows since both model structures have objectwise weak equivalences.
\end{proof}

\begin{thm} \label{thm:comparison-realization-2}
	Assume that \(\A\) has cofibrant representables. Let \[\wh A \quad\hbox{and}\quad\wh B\] be two realization functors over \(I\). Then there exists a realization functor \(\wh C\) over \(I\) and a commutative diagram of \(\mathcal M^\A\) 
	\[
	\begin{tikzcd}[column sep=6em,row sep=3em]
		A \arrow[d,twoheadrightarrow,"\simeq"] &\arrow[l,"\mu"']C \arrow[r,"\nu"] \arrow[d,twoheadrightarrow,"\simeq"]& B \arrow[d,twoheadrightarrow,"\simeq"]\\
		I& I \arrow[l,equal]\arrow[r,equal] & I
	\end{tikzcd}
	\]
	giving rise to two natural maps of \(\mathcal{M}\)
	\[
	\begin{tikzcd}[column sep=large]
		{\wh{A}(K)}
		&
		{\wh{C}(K)}
		\arrow[l, "{\wh\mu_K}"']
		\arrow[r, "{\wh\nu_K}"]
		&
		{\wh{B}(K)}
	\end{tikzcd}
	\]
	and, for every cofibrant \(\A\)-set \(K\), these two maps are weak equivalences of \(\mathcal{M}\) between cofibrant objects.
\end{thm}

\begin{proof}
	Consider the diagram of \(\mathcal M^\A\)
	\[
	\begin{tikzcd}[column sep=4em,row sep=4em]
		C \arrow[d,"\simeq"',twoheadrightarrow]\arrow[r,dashed,"\ell"] & B \arrow[d,twoheadrightarrow,"\simeq"]\\
		A \arrow[r,twoheadrightarrow,"\simeq"] & I
	\end{tikzcd}
	\]
	where \(C\) is a cofibrant replacement of \(A\) in the projective model structure of \(\mathcal M^\A\). The lift \(\ell\) exists since \(C\) is cofibrant and since the map \(B\to I\) is a trivial projective fibration. By Theorem~\ref{thm:proj-cof-suff-cond}, the map \[\wh C(\partial\A[-]) \longrightarrow \wh C(\A[-])\] is a projective cofibration of \(\mathcal{M}^{\A_{=n}}\). By Lemma~\ref{lem:proj-inj}, the map \[\wh C(\partial\A(-,x)) \longrightarrow \wh C(\A(-,x))\] is therefore a cofibration of \(\mathcal{M}\) for all \(x\in \Ob(\A_{=n})\). Thus \(\wh C\) is a realization functor over \(I\). The proof is complete by Theorem~\ref{thm:comparison-realization-1}. 	
\end{proof}

\section{Thick categories of cubes}
\label{sec:thickness}

We now apply the categorical results to categories of cubes, graded by dimension. To use the realization comparison, we first study when free extension carries the boundary of a precubical cube to the boundary of the corresponding representable.

Let
\[
[0]=\{()\},
\qquad
[n]=\{0<1\}^n \quad (n\geq 1),
\]
equipped with the product order.  We write
\[
0_n=(0,\ldots,0),
\qquad
1_n=(1,\ldots,1).
\]
Let \(\PoSet^+\) denote the category of posets and strictly increasing maps.
For vertices \(x=(x_1,\ldots,x_n)\) and \(y=(y_1,\ldots,y_n)\) of \([n]\), put
\[
\dd(x,y)=
\begin{cases}
	\displaystyle\sum_{i=1}^n (y_i-x_i), & x\leq y,\\[5mm]
	+\infty, & \text{otherwise.}
\end{cases}
\]
Thus \(\dd(x,y)=1\) means that \(x<y\) and the two vertices differ in exactly one coordinate. The same formula for \(x,y\in[0,1]^n\) defines a Lawvere metric on the topological \(n\)-cube by \cite[Proposition~1.5]{DirectedDegeneracy}. For \(n\geq 1\), \(1\leq i\leq n\), and \(\alpha\in\{0,1\}\), the \emph{coface map}
\[
\delta_i^\alpha:[n-1]\longrightarrow[n]
\]
is defined by
\[
\delta_i^\alpha(x_1,\ldots,x_{n-1})
=
(x_1,\ldots,x_{i-1},\alpha,x_i,\ldots,x_{n-1}).
\]

\begin{definition}
	The \emph{box category} \(\Sq\) is the subcategory of \(\PoSet^+\) generated by the coface maps \(\delta_i^\alpha\). The presheaves on the box category are called \textit{precubical sets}.
\end{definition}

\begin{definition} \cite[Definition~2.1.5]{symcub}
	A strictly increasing map
	\[
	f:[m]\longrightarrow[n]
	\]
	is \emph{cotransverse} if it preserves adjacency in the following sense:
	\[
	\dd(x,y)=1
	\quad\Longrightarrow\quad
	\dd(f(x),f(y))=1
	\qquad (x,y\in[m]).
	\]
	The category \(\SqhatS\) has the cubes \([n]\), \(n\geq 0\), as objects and all cotransverse maps as morphisms:
	\[
	\SqhatS([m],[n])
	=
	\{\text{cotransverse maps }[m]\to[n]\}.
	\]
	The presheaves on \(\SqhatS\) are called \textit{(symmetric) transverse sets}.
\end{definition}

An immediate consequence of strict monotonicity is the following important fact:

\begin{proposition} \label{prop:cotransverse-endo-endpoint-preserving}
	For \(n\geq 0\), every endomorphism of \([n]\) preserves \(0_n\) and \(1_n\):
	\[
	\hbox{For all}\quad n\geq 0,\quad \theta\in \SqhatS([n],[n]) \quad \Longrightarrow \quad \theta(0_n)=0_n \quad \hbox{and} \quad \theta(1_n)=1_n\,.
	\]
\end{proposition}

\begin{proof}
Indeed, let
\[
0_n=x_0<x_1<\cdots <x_n=1_n
\]
be a maximal chain in \([n]\). Since \(\theta\) is strictly increasing, the sequence
\[
\theta(x_0)<\theta(x_1)<\cdots <\theta(x_n)
\]
is a strict chain of length \(n\) in \([n]\). If \(r(y)=\dd(0_n,y)\) denotes the rank of a vertex
\(y\), then
\[
r(\theta(x_i))\geq r(\theta(x_0))+i
\]
for every \(i\). Taking \(i=n\), we get
\[
n\geq r(\theta(x_n))\geq r(\theta(x_0))+n.
\]
Thus \(r(\theta(x_0))=0\), and therefore \(\theta(0_n)=0_n\). The same inequalities also force \(r(\theta(x_n))=n\), and therefore
\(\theta(1_n)=1_n\).
\end{proof}

\begin{definition} [{\cite[Definition~2.1.7 and 2.1.12]{symcub}}] \label{def:category-cubes}
	A \emph{category of cubes} is a small category \(\A\) satisfying
	\[
	\Sq\subset \A\subset \SqhatS,
	\]
	where all three categories have the same objects \([n]\), \(n\geq 0\), and the inclusions are the identity on objects.
\end{definition}

\begin{proposition}[cotransverse--cocubical factorization {\cite[Proposition~3.1.14]{symcub}}]  \label{prop:decomposition_distance}
	Every cotransverse map
	\[
	f:[m]\longrightarrow[n]
	\]
	factors uniquely as
	\[
	[m]\xrightarrow{\ \psi\ }[m]\xrightarrow{\ \delta\ }[n],
	\qquad
	f=\delta\psi,
	\]
	where \(\delta\in\Sq([m],[n])\) and \(\psi\in\SqhatS([m],[m])\) is a cotransverse endomorphism. 
\end{proposition}

\begin{definition} [{\cite[Definition~2.10]{ThickCubes}}] \label{def:thickness}
	A category of cubes \(\A\) is \emph{thick} if the preceding cotransverse--cocubical factorization is internal to \(\A\): whenever
	\[
	f:[m]\longrightarrow[n]
	\]
	is a morphism of \(\A\) and
	\[
	f=\delta\psi
	\]
	is its standard factorization with \(\delta\in\Sq([m],[n])\) and \(\psi\in\SqhatS([m],[m])\), then
	\[
	\psi\in\A([m],[m]).
	\]
	Since \(\Sq\subset\A\), this implies that both factors \(\delta\) and \(\psi\) are morphisms of \(\A\).
\end{definition}

Every category of cubes is c-direct with the degree function \(d([n])=n\) for \(n\geq 0\). Let
\[
i:\Sq\subset \A
\]
be the inclusion of the box category into a category of cubes. We denote by
\[
\omega_{\A}:\A^\op\Set\longrightarrow \Sq^\op\Set
\]
the restriction functor, and by
\[
\LA:\Sq^\op\Set\longrightarrow \A^\op\Set
\]
its left adjoint, that is, the left Kan extension along \(i^\op:\Sq^\op\to \A^\op\). The canonical map 
\[
\LA(\Sq[n]) \longrightarrow \A[n]
\]
is an isomorphism of \(\A\)-sets for any category of cubes \(\A\) by \cite[Proposition~2.12]{ThickCubes}.

\begin{thm} \label{thm:thick-caracterization-1}
	Let \(\A\) be a category of cubes. Then \(\A\) is thick if and only if, for every \(n\geq 1\), the canonical comparison map
	\[
	\LA(\partial\Sq[n])\longrightarrow \partial\A[n]
	\]
	is an isomorphism of \(\A\)-sets.
\end{thm}

\begin{proof}
	If \(\A\) is thick, the assertion follows from \cite[Proposition~2.12]{ThickCubes}, which proves the isomorphism
	\[
	\LA(\partial\Sq[n])\cong \partial\A[n]
	\]
	for all \(n\geq 0\), hence in particular for all \(n\geq 1\). Conversely, suppose that \(\A\) is not thick. By the characterization of thickness recalled above, there exists a morphism
	\[
	f:[m]\longrightarrow [n]
	\]
	of \(\A\) whose unique cotransverse--cocubical factorization
	\[
	f=\delta\phi,
	\qquad
	\phi:[m]\to[m],
	\qquad
	\delta:[m]\to[n],
	\]
	with \(\delta\in\Sq([m],[n])\) and \(\phi\) cotransverse, satisfies
	\[
	\phi\notin \A([m],[m]).
	\]
	Choose such an \(f\) with \(n-m\) minimal. We first observe that \(m<n\). Indeed, if \(m=n\), then the cubical factor \(\delta\) is the identity of \([m]\), so that \(f=\phi\). Since \(f\in\A([m],[m])\), this would imply \(\phi\in\A([m],[m])\), a contradiction. Hence \(m<n\), and in particular \(n\geq 1\). Since \(m<n\), the morphism \(f\) defines an element
	\[
	f\in \partial\A[n]([m])=\A([m],[n]).
	\]
	We claim that this element is not in the image of
	\[
	\LA(\partial\Sq[n])([m])\longrightarrow \partial\A[n]([m]).
	\]
	Indeed, by the coend formula for the left Kan extension,
	\[
	\LA(\partial\Sq[n])([m])
	\cong
	\int^{[q]\in\Sq}
	\partial\Sq[n]([q])\times \A([m],[q]).
	\]
	Since \(\partial\Sq[n]([q])\) is nonempty only for \(q<n\), every element of \(\LA(\partial\Sq[n])([m])\) is represented by a pair
	\[
	[m]\xrightarrow{g}[q]\xrightarrow{\alpha}[n]
	\]
	with
	\[
	g\in\A([m],[q]),
	\qquad
	\alpha\in\Sq([q],[n]),
	\qquad
	q<n.
	\]
	The comparison map sends the class of this pair to the composite \(\alpha g\in\A([m],[n])\). Suppose, for a contradiction, that \(f\) belongs to the image of the comparison map. Then there exist \(q<n\), a morphism \(g\in\A([m],[q])\), and a cubical map \(\alpha\in\Sq([q],[n])\), such that
	\[
	f=\alpha g.
	\]
	Factor \(g\) uniquely as
	\[
	g=\epsilon\psi,
	\qquad
	\psi:[m]\to[m],
	\qquad
	\epsilon:[m]\to[q],
	\]
	with \(\epsilon\in\Sq([m],[q])\) and \(\psi\) cotransverse. Then
	\[
	f=\alpha\epsilon\psi.
	\]
	Since \(\alpha\epsilon\in\Sq([m],[n])\), uniqueness of the cotransverse--cocubical factorization of \(f\) gives
	\[
	\psi=\phi
	\qquad\text{and}\qquad
	\alpha\epsilon=\delta.
	\]
	Consequently
	\[
	g=\epsilon\phi.
	\]
	Thus \(g\in\A([m],[q])\) is another morphism of \(\A\) whose cotransverse endomorphism part is \(\phi\notin\A([m],[m])\). Since \(q<n\), we have
	\[
	q-m<n-m,
	\]
	contradicting the minimality of \(f\). Therefore \(f\) is not in the image of
	\[
	\LA(\partial\Sq[n])([m])\longrightarrow \partial\A[n]([m]).
	\]
	The comparison map is not surjective at \([m]\), and hence it is not an isomorphism of \(\A\)-sets. Thus, if the comparison map is an isomorphism for all \(n\geq 1\), then \(\A\) must be thick.
\end{proof}

\begin{thm} \label{thm:thick-notcaracterization}
	Let \(\A\) be a category of cubes. If \(\A\) is thick, then \(\A(-,[n])\) is cofibrant for every \(n\geq 0\).
\end{thm}

\begin{proof}
	We prove, by induction on \(n\), that if \(\A\) is thick, then \(\A(-,[n])\) is cofibrant for all \(n\geq 0\). For \(n=0\), the boundary \(\partial\A(-,[0])\) is empty, so \(\A(-,[0])\) is obtained from the initial presheaf by attaching one \(0\)-cell. Assume now that \(\A(-,[m])\) is cofibrant for all \(m<n\). Since \(\partial\Sq[n]\) is obtained in \(\Sq^\op\Set\) by attaching the ordinary lower-dimensional cubical faces \(\Sq[m]\) with \(m<n\), its image \(\LA(\partial\Sq[n])\) is obtained from the presheaves \(\A(-,[m])\), \(m<n\), by the corresponding finite colimit construction. By the induction hypothesis, these presheaves are cofibrant, and cofibrant presheaves are closed under such cellular colimits. Since \(\A\) is thick, we have
\[
\LA(\partial\Sq[n])
\cong
\partial\A(-,[n])
\]
by Theorem~\ref{thm:thick-caracterization-1}. Hence \(\partial\A(-,[n])\) is cofibrant. Therefore \(\A(-,[n])\) is cofibrant, because it is obtained from \(\partial\A(-,[n])\) by attaching one \(n\)-cell along the generating cofibration
\[
\partial\A(-,[n])
\subseteq
\A(-,[n]).
\]
\end{proof}

\begin{remark}
	The converse of the preceding statement is false in general. Let
\[
\gamma^-:[2]\longrightarrow [2],
\qquad
\gamma^-(x,y)=(\min(x,y),\max(x,y)),
\]
and let
\[
u=\delta^0_3\gamma^-:[2]\longrightarrow [3].
\]
Let \(\A\) be the subcategory of \(\widehat{\square}_S\) generated by the box category \(\square\) and by the single morphism \(u\). Then \(\A\) is a category of cubes. It is not thick: indeed, the standard cotransverse--cocubical factorization of \(u\) is
\[
[2]\xrightarrow{\gamma^-}[2]\xrightarrow{\delta^0_3}[3],
\]
but \(\gamma^-\notin \A([2],[2])\). To see this, observe that every nonidentity generator of \(\A\) strictly raises degree. Hence every nonidentity morphism of \(\A\) strictly raises degree, and in particular \(\A([2],[2])=\{\id_{[2]}\}\). Nevertheless, all representable presheaves of \(\A\) are cofibrant. Indeed, the same observation shows that \(\A\) is a direct category for the degree \(d([n])=n\). For a direct category, each representable is cellular: the presheaf \(\A(-,[n])\) is obtained by attaching, in increasing degree \(p\leq n\), one \(p\)-cell for each morphism \([p]\to [n]\) of \(\A\). The boundary of such a cell only involves morphisms from objects of degree \(<p\), and therefore has already been attached at earlier stages. Thus \(\A(-,[n])\) is cellular, hence cofibrant, for every \(n\geq 0\), although \(\A\) is not thick.
\end{remark}

\section{Branching spaces of transverse sets}
\label{sec:branching}

\begin{notation}
	A thick category of cubes \(\A\) is fixed for the rest of the paper. 
\end{notation}

Let \(\A^-\) be the subcategory of \(\A\) with the same objects and whose morphisms are the maps
\[
f:[m]\longrightarrow [n] \qquad \hbox{such that}\qquad f(0_m)=0_n.
\]
The small category \(\A^-\) is clearly c-direct. Let \(\Sq^{-}\) denote the wide subcategory of \(\Sq\) consisting of the cubical maps preserving the initial vertex. Equivalently, \(\Sq^{-}\) is generated by the coface maps \(\delta_i^0\). Thus a morphism \(\delta:[m]\to[n]\) of \(\Sq\) belongs to \(\Sq^{-}\) if and only if \(\delta(0_m)=0_n\). Let
\[
j:\Sq^{-}\longrightarrow \A^{-}
\]
be the inclusion. We denote by
\[
\LA^{-}:(\Sq^{-})^{\op}\Set
\longrightarrow
(\A^{-})^{\op}\Set
\]
the left Kan extension along \(j^{\op}\). Thus
\[
\LA^{-}(\Sq^{-}(-,[n]))
\cong
\A^{-}(-,[n])
\]
for every \(n\geq 0\).

\begin{lem} \label{lem:transverse-box-minus}
	The cotransverse--cocubical factorization of every morphism of \(\A^{-}\) is internal to \(\A^{-}\) after replacing \(\Sq\) by \(\Sq^{-}\). More explicitly, if
	\[
	f:[m]\longrightarrow[n]
	\]
	is a morphism of \(\A^{-}\), and if
	\[
	f=\delta\psi
	\]
	is its cotransverse--cocubical factorization, with
	\[
	\psi:[m]\longrightarrow[m]
	\qquad\text{and}\qquad
	\delta:[m]\longrightarrow[n]
	\]
	where \(\delta\in\Sq([m],[n])\), then
	\[
	\delta\in\Sq^{-}([m],[n])
	\qquad\text{and}\qquad
	\psi\in\A^{-}([m],[m]).
	\]
\end{lem}

\begin{proof}
	Since \(f\) belongs to \(\A^{-}\), one has \(f(0_m)=0_n\). Hence
	\[
	\delta(\psi(0_m))=0_n.
	\]
	Since \(0_m\leq \psi(0_m)\), we have
	\[
	\delta(0_m)\leq \delta(\psi(0_m))=0_n.
	\]
	As \(0_n\) is the least vertex of \([n]\), this implies
	\[
	\delta(0_m)=0_n.
	\]
	Thus \(\delta\in\Sq^{-}([m],[n])\). Moreover, if \(\psi(0_m)\neq 0_m\), then \(0_m<\psi(0_m)\), and the strict monotonicity of \(\delta\) gives
	\[
	\delta(0_m)<\delta(\psi(0_m))=0_n,
	\]
	which is impossible. Therefore \(\psi(0_m)=0_m\). Since \(\A\) is thick, the cotransverse endomorphism \(\psi\) belongs to \(\A([m],[m])\). Together with
	\(\psi(0_m)=0_m\), this gives
	\[
	\psi\in\A^{-}([m],[m]).
	\]
\end{proof}

For \(n\geq 0\), let
\[
\partial\Sq^{-}[n]\subseteq \Sq^{-}(-,[n])
\]
be the strict lower-degree part of the representable presheaf, namely
\[
\partial\Sq^{-}[n]([m])
=
\begin{cases}
	\Sq^{-}([m],[n]), & m<n,\\
	\varnothing, & m\geq n.
\end{cases}
\]
Likewise, \(\partial\A^{-}[n]\) denotes the strict lower-degree part of \(\A^{-}(-,[n])\).

\begin{lem} \label{lem:Aminus-boundary}
	For every \(n\geq 0\), the canonical comparison map
	\[
	\LA^{-}(\partial\Sq^{-}[n])
	\longrightarrow
	\partial\A^{-}[n]
	\]
	is an isomorphism of \(\A^{-}\)-sets.
\end{lem}

\begin{proof}
	We evaluate the comparison map at \([r]\). By the coend formula for the left Kan extension, one has
	\[
	\LA^{-}(\partial\Sq^{-}[n])([r])
	\cong
	\int^{[q]\in\Sq^{-}}
	\partial\Sq^{-}[n]([q])\times \A^{-}([r],[q]).
	\]
	Since \(\partial\Sq^{-}[n]([q])\) is nonempty only for \(q<n\), every element is represented by a pair
	\[
	[r]\xrightarrow{g}[q]\xrightarrow{\alpha}[n],
	\]
	with
	\[
	g\in\A^{-}([r],[q]),
	\qquad
	\alpha\in\Sq^{-}([q],[n]),
	\qquad
	q<n.
	\]
	The comparison map sends the class of such a pair to the composite
	\[
	\alpha g\in \A^{-}([r],[n]).
	\]
	Since \(g\) exists only if \(r\leq q\), this composite belongs to \(\partial\A^{-}[n]([r])\). We first prove surjectivity. Let
	\[
	f\in \partial\A^{-}[n]([r]).
	\]
	Thus \(r<n\), and \(f:[r]\to[n]\) is a morphism of \(\A^{-}\). Write the standard
	cotransverse--cocubical factorization of \(f\) as
	\[
	f=\delta\psi,
	\]
	with \(\delta\in\Sq([r],[n])\) and \(\psi:[r]\to[r]\) cotransverse. By Lemma~\ref{lem:transverse-box-minus},
	\[
	\delta\in\Sq^{-}([r],[n])
	\qquad\text{and}\qquad
	\psi\in\A^{-}([r],[r]).
	\]
	Since \(r<n\), the morphism \(\delta\) is an element of
	\(\partial\Sq^{-}[n]([r])\). Hence \(f\) is the image of the class represented by
	\[
	[r]\xrightarrow{\psi}[r]\xrightarrow{\delta}[n].
	\]
	We now prove injectivity. Suppose that two representatives
	\[
	[r]\xrightarrow{g}[q]\xrightarrow{\alpha}[n],
	\qquad
	[r]\xrightarrow{h}[q']\xrightarrow{\beta}[n],
	\]
	with \(\alpha,\beta\) in \(\Sq^{-}\), have the same image:
	\[
	\alpha g=\beta h.
	\]
	Factor \(g\) and \(h\) by their cotransverse--cocubical factorizations:
	\[
	g=\epsilon\psi,
	\qquad
	h=\eta\chi,
	\]
	with
	\[
	\epsilon\in\Sq^{-}([r],[q]),
	\quad
	\eta\in\Sq^{-}([r],[q']),
	\quad
	\psi,\chi\in\A^{-}([r],[r]),
	\]
	again by the preceding lemma. In the coend, the defining relations give
	\[
	[\alpha,g]=[\alpha\epsilon,\psi],
	\qquad
	[\beta,h]=[\beta\eta,\chi].
	\]
	The equality \(\alpha g=\beta h\) becomes
	\[
	\alpha\epsilon\psi=\beta\eta\chi.
	\]
	Here \(\alpha\epsilon\) and \(\beta\eta\) are morphisms of \(\Sq^{-}\). By uniqueness of the standard cotransverse--cocubical factorization, we obtain
	\[
	\alpha\epsilon=\beta\eta
	\qquad\text{and}\qquad
	\psi=\chi.
	\]
	Therefore
	\[
	[\alpha,g]=[\alpha\epsilon,\psi]
	=
	[\beta\eta,\chi]=[\beta,h].
	\]
	This proves injectivity, and hence the comparison map is an isomorphism.
\end{proof}

\begin{thm} \label{thm:Aminus-representablecofibrant}
	For every \(n\geq 0\), the representable presheaf
	\[
	\A^{-}(-,[n]) : (\A^{-})^{\op}\longrightarrow\Set
	\]
	is cofibrant. Equivalently, \(\A^{-}\) has cofibrant representables.
\end{thm}

\begin{proof}
	We prove, by induction on \(n\), that \(\A^{-}(-,[n])\) is cellular, hence cofibrant. For \(n=0\), the boundary \(\partial\A^{-}[0]\) is empty. Therefore \(\A^{-}(-,[0])\) is obtained from the initial presheaf by attaching one \(0\)-cell. Assume now that \(\A^{-}(-,[m])\) is cellular for every \(m<n\). The presheaf \(\partial\Sq^-[n]\) is obtained in \((\Sq^-)^{\op}\Set\) by the finite cellular construction obtained by attaching the proper initial-vertex-preserving faces of \([n]\). Equivalently, it is built using only cells \(\Sq^-(-,[m])\) with \(m<n\), attached along their boundaries \(\partial\Sq^-[m]\). Applying the colimit-preserving functor \(\LA^{-}\), and using the preceding lemma for all \(m\leq n\), this gives a cellular construction of
	\[
	\LA^{-}(\partial\Sq^{-}[n])
	\cong
	\partial\A^{-}[n]
	\]
	using only the generating cofibrations
	\[
	\partial\A^{-}[m]\subseteq \A^{-}(-,[m]),
	\qquad m<n.
	\]
	Hence \(\partial\A^{-}[n]\) is cellular. Finally, \(\A^{-}(-,[n])\) is obtained from \(\partial\A^{-}[n]\) by attaching one \(n\)-cell along the generating cofibration
	\[
	\partial\A^{-}[n]\subseteq \A^{-}(-,[n]).
	\]
	Therefore \(\A^{-}(-,[n])\) is cellular. This completes the induction.
\end{proof}

A \textit{directed path} in \([0,1]^n\) starting from \(0_n\) is a continuous map \(\gamma=(\gamma_1,\dots,\gamma_n)\) from \([0,\ell]\) to \([0,1]^n\) such that \(\gamma(0)=0_n\) and each \(\gamma_i\) is nondecreasing. It is \textit{natural} if (see \cite[Section~2.2]{MR2521708} and \cite[Definition~4.8]{NaturalRealization})
\[
t=\dd(0_n,\gamma(t))\qquad \hbox{for all}\qquad t\in [0,\ell].
\]
Let \(n\geq 1\). A \textit{short natural directed path} of \([0,1]^n\) is a natural directed path \(\phi:[0,\epsilon]\to [0,1]^n\) such that \(\phi(0)=0_n\) and \(0<\epsilon<1\). The real number \(\epsilon\) is called the \textit{natural length} of the short natural directed path. The set of short natural directed paths of \([0,1]^n\) of natural length \(\epsilon\in ]0,1[\) is denoted by \(N_n(\epsilon)\). We set \(N_0(\epsilon)=\varnothing\).  For \(n\geq 1\), by \cite[Proposition~4.7]{cubical-branching}, the set \(N_n(\epsilon)\) equipped with the compact-open topology is \(\Delta\)-generated, \(\Delta\)-Hausdorff, metrizable, contractible, compact and sequentially compact.

\begin{definition} [{\cite[Definition~3.2]{DirectedDegeneracy}}] \label{def:top-cotransverse}
	Let \(f=(f_1,\dots,f_n):[n]\to [n]\) be a cotransverse map. Let \[\T(f):[0,1]^n\to [0,1]^n\] be the continuous map defined by 
	\[
	\T(f)(x_1,\dots,x_n) = (\T(f)_1(x_1,\dots,x_n),\dots,\T(f)_n(x_1,\dots,x_n))
	\]
	with \[\T(f)_i(x_1,\dots,x_n) = \max_{(\epsilon_1,\dots,\epsilon_n)\in f_i^{-1}(1)} \min \{x_k\mid \epsilon_k=1\}\] for all \(1\leq i\leq n\).
\end{definition}

\begin{notation} [{\cite[Notation~3.6]{DirectedDegeneracy}}]
	For \(\delta_i^\alpha:[n-1]\to [n] \in \Sq\), let \[\T(\delta_i^\alpha)=
	\begin{cases}
		[0,1]^{n-1} \to [0,1]^n\\
		(\epsilon_1, \dots, \epsilon_{n-1})\mapsto (\epsilon_1,\dots, \epsilon_{i-1}, \alpha, \epsilon_i, \dots, \epsilon_{n-1})
	\end{cases}
	\] for all \(n\geq 1\) and \(\alpha\in\{0,1\}\).
\end{notation}

If
\[
f:[m]\longrightarrow [n]
\]
is a morphism of \(\A\), we denote by
\[
\T(f):[0,1]^m\longrightarrow [0,1]^n
\]
the associated continuous map of topological cubes, as defined above. Together with the identities \(\T(\id_{[n]})=\id_{[0,1]^n}\) and \(\T(gf)=\T(g)\T(f)\) \cite[Proposition~3.8]{DirectedDegeneracy}, this defines a functor from \(\A\) to \(\Top\), sending \([n]\) to \([0,1]^n\), where, in this paper, \(\Top\) denotes either the category of \(\Delta\)-generated spaces or the category of \(\Delta\)-Hausdorff \(\Delta\)-generated spaces. 

\begin{definition}
For an \(\A\)-set \(K\), write \(K_n=K([n])\) and set
\[
|K|_{\geom} = \int^{[n]\in \A} K_n\cdot \T([n])\,.
\]
This gives rise to a colimit-preserving functor \(|-|_{\geom}:\A^{\op}\Set \to \Top\) called the \textit{geometric realization}. 
\end{definition}

\begin{proposition}
	Let \(\epsilon\in ]0,1[\). For any \(u\in N_m(\epsilon)\) and any morphism \(f:[m]\longrightarrow [n]\) of \(\A^-\), one has \(\T(f)u \in N_n(\epsilon)\). Equivalently, 
	\[
	\T(f)(N_m(\epsilon))\subseteq N_n(\epsilon).
	\]
\end{proposition}

\begin{proof}
	By \cite[Corollary~3.15]{DirectedDegeneracy}, for any cotransverse map \(f:[m]\longrightarrow [n]\), the continuous map 
	\[
	\T(f):[0,1]^m\longrightarrow [0,1]^n
	\]
	induces a map of Lawvere metric spaces which is also quasi-isometric for the Lawvere distance \(\dd\). Since \(f\in\A^-\), we have \(\T(f)(0_m)=0_n\). Thus for any \(u\in N_m(\epsilon)\), one has for all \(t\in[0,\epsilon]\)
	\[t=\dd(0_m,u(t)) = \dd(0_n,\T(f)(u(t))). \]
	Using \cite[Proposition~4.15]{DirectedDegeneracy}, we deduce that \(\T(f)u\) is natural.
\end{proof}

Hence we obtain a functor 
\[
N(\epsilon):\A^-\longrightarrow \Top,
\qquad
[n]\longmapsto N_n(\epsilon).
\] 
Let
\[
\rho_\A:\A^{\op}\Set\longrightarrow(\A^-)^{\op}\Set
\]
denote the restriction functor along the inclusion \(\A^-\subseteq\A\). Since colimits of presheaves are computed pointwise, \(\rho_\A\) preserves all colimits.

\begin{lem} \label{lem:rho-preserves-cofibrant}
	The restriction functor
	\[
	\rho_\A:\A^{\op}\Set\longrightarrow(\A^-)^{\op}\Set
	\]
	preserves cofibrant presheaves.
\end{lem}

\begin{proof}
	Since \(\rho_\A\) preserves colimits and retracts, it suffices to prove that it sends each generating cofibration
	\[
	\partial\A[n]\subseteq \A(-,[n])
	\]
	to a cofibration of \((\A^-)^{\op}\Set\). Consider the commutative square
	\[
	\begin{tikzcd}[column sep=3em,row sep=3em]
		\partial\A^-[n] \arrow[r] \arrow[d]
		&
		\rho_\A(\partial\A[n]) \arrow[d]
		\\
		\A^-(-,[n]) \arrow[r]
		&
		\rho_\A(\A(-,[n])) .
	\end{tikzcd}
	\]
	We claim that this square is a pushout. Evaluating at \([q]\), there are three cases. 
	\begin{enumerate}
		\item If \(q<n\), then both \(\rho_\A(\partial\A[n])([q])\) and \(\rho_\A(\A(-,[n]))([q])\) are equal to \(\A([q],[n])\), so the assertion is immediate.
		\item If \(q>n\), then all terms are empty, since \(\A\) is c-direct.
		\item If \(q=n\), then
		\[
		\partial\A[n]([n])=\varnothing
		\qquad\text{and}\qquad
		\partial\A^-[n]([n])=\varnothing.
		\]
		Moreover
		\[
		\rho_\A(\A(-,[n]))([n])=\A([n],[n])
		\]
		and
		\[
		\A^-(-,[n])([n])=\A^-([n],[n]).
		\]
		These two sets are equal since every endomorphism of \([n]\) preserves \(0_n\) by Proposition~\ref{prop:cotransverse-endo-endpoint-preserving}. Hence the square is a pushout objectwise, and therefore a pushout of presheaves.
	\end{enumerate}
	Thus \(\rho_\A(\partial\A[n]\subseteq\A(-,[n]))\) is a pushout of the generating cofibration
	\[
	\partial\A^-[n]\subseteq\A^-(-,[n]).
	\]
	It is therefore a cofibration. Hence \(\rho_\A\) sends cellular \(\A\)-sets to cellular \(\A^-\)-sets, and since it also preserves retracts, it sends cofibrant \(\A\)-sets to cofibrant \(\A^-\)-sets.
\end{proof}

Define
\[
\widehat N_\epsilon:(\A^-)^{\op}\Set\longrightarrow\Top
\]
by
\[
\widehat N_\epsilon(H)
=
\int^{[n]\in\A^-} H([n])\cdot N_n(\epsilon).
\]

\begin{definition} \label{def:branching-space}
	For an \(\A\)-set \(K\), set
\[
B^-_\epsilon(K)
=
\widehat N_\epsilon(\rho_\A K).
\]
Equivalently,
\[
B^-_\epsilon(K)
=
\int^{[n]\in\A^-} K_n\cdot N_n(\epsilon).
\]
The topological space \(B^-_\epsilon(K)\) is called the \textit{\(\epsilon\)-branching space} of \(K\).
\end{definition}

\begin{thm} \label{thm:Bminus-realization}
	Let \(\mathbf{1}_\A\in \Top^{\A^-}\) be the unique functor such that \[\mathbf{1}_\A([0])=\varnothing\quad \hbox{and}\quad \mathbf{1}_\A([n])=\{n\} \quad \hbox{for} \quad n\geq 1.\] Let \(\epsilon\in]0,1[\). The functor
	\[
	\widehat N_\epsilon:(\A^-)^{\op}\Set\longrightarrow\Top
	\]
	is a realization functor over \(\mathbf{1}_\A\) for \(\Top\) equipped with the m-model structure or with the h-model structure.
\end{thm}

\begin{proof}
	For \(n=0\), the Yoneda lemma gives
	\[
	\widehat N_\epsilon(\A^- [0])\cong N_0(\epsilon)=\varnothing
	=\mathbf{1}_\A([0]),
	\]
	so the canonical map
	\[
	\widehat N_\epsilon(\A^- [0])\longrightarrow \mathbf{1}_\A([0])
	\]
	is an isomorphism. For \(n\geq 1\), the Yoneda lemma gives
	\[
	\widehat N_\epsilon(\A^-[n])\cong N_n(\epsilon).
	\]
	The space \(N_n(\epsilon)\) is contractible. Hence the canonical map
	\[
	\widehat N_\epsilon(\A^-[n])\longrightarrow \mathbf{1}_\A([n])
	\]
	is a trivial h-fibration and a trivial m-fibration for \(n\geq 1\). Together with the case \(n=0\), this proves condition~(2) of Definition~\ref{def:realization}. By Theorem~\ref{thm:Aminus-representablecofibrant}, the category \(\A^-\) has cofibrant representables. It remains to verify the cofibration condition. By Lemma~\ref{lem:Aminus-boundary},
	\[
	\LA^-(\partial\Sq^-[n])\cong \partial\A^-[n].
	\]
	For \(n=0\), both sides are empty and the corresponding map is the identity of
	\(\varnothing\).

	Assume now that \(n\geq 1\). By the preceding isomorphism and the co-Yoneda lemma, \(\widehat N_\epsilon(\partial\A^-[n])\) is the colimit over the proper faces of \([0,1]^n\) containing \(0_n\), with \(N_q(\epsilon)\) assigned to each face of dimension \(q\). The faces of codimension one give the subspaces
	\[
	F_i=\{u\in N_n(\epsilon)\mid u_i\equiv 0\},\qquad 1\leq i\leq n,
	\]
	and their intersections are the path spaces of the common faces. Hence the canonical map from this finite colimit to \(\bigcup_{i=1}^n F_i\) is a continuous bijection. It is a homeomorphism since the face spaces are compact and \(N_n(\epsilon)\) is Hausdorff. This union is precisely the space of short natural directed paths in the boundary starting from \(0_n\). Indeed, since \(\epsilon<1\), a path ending in the boundary has a zero coordinate at its endpoint, and this coordinate is identically zero by monotonicity. The inclusion
	\[
	\widehat N_\epsilon(\partial\A^-[n])
	\subseteq
	\widehat N_\epsilon(\A^-[n])
	\]
	is an m-cofibration by \cite[Proposition~6.1]{cubical-branching}, and therefore also an h-cofibration. Thus \(\widehat N_\epsilon\) is a realization functor over \(\mathbf{1}_\A\).
\end{proof}

\begin{cor} \label{cor:Bminus-invariance}
	Let \(\epsilon,\epsilon'\in]0,1[\). There exists a zigzag of natural transformations
	\[
	B^-_\epsilon \;\longleftarrow\; \bullet \;\longrightarrow\; B^-_{\epsilon'}
	\]
	such that, for every cofibrant \(\A\)-set \(K\), the two maps are homotopy equivalences between m-cofibrant spaces.
\end{cor}

\begin{proof}
	By Theorem~\ref{thm:Bminus-realization} and Theorem~\ref{thm:comparison-realization-2} applied to the c-direct category \(\A^-\), the corresponding statement holds for every cofibrant \(\A^-\)-set. If \(K\) is cofibrant as an \(\A\)-set, then \(\rho_\A K\) is cofibrant as an \(\A^-\)-set by Lemma~\ref{lem:rho-preserves-cofibrant}. Since
	\[
	B^-_\epsilon(K)=\widehat N_\epsilon(\rho_\A K),
	\]
	the result follows. Finally, weak homotopy equivalences between m-cofibrant spaces are homotopy equivalences by \cite[Corollary~3.4]{mixed-cole}.
\end{proof}

Corollary~\ref{cor:Bminus-invariance} applies in particular to every free \(\A\)-set of the form \(\LA(K)\), where \(K\) is a precubical set. Indeed, every precubical set is cofibrant as a \(\Sq\)-set, and, since \(\A\)
is thick, the functor \(\LA\) sends the generating cofibrations
\[
\partial\Sq[n]\subseteq \Sq[n]
\]
to
\[
\partial\A[n]\subseteq \A[n].
\]
Thus \(\LA(K)\) is cofibrant as an \(\A\)-set.

For the particular diagram \(N(\epsilon)\), a change of length also admits an explicit comparison on every \(\A\)-set. Given \(0<\epsilon,\epsilon'<1\), put \(r=\epsilon'/\epsilon\). For \(n\geq 1\) the map
\[
S_r:N_n(\epsilon)\longrightarrow N_n(\epsilon'),\qquad
(S_r u)(t)=r\,u(t/r)\quad (0\leq t\leq\epsilon')
\]
is a homeomorphism. Indeed, the coordinates of \(S_r u\) are nondecreasing and sum to \(t\), so each lies in \([0,\epsilon']\). The uniform distance is multiplied by \(r\), and the inverse is \(S_{1/r}\). In degree zero both spaces are empty. The minima and maxima in Definition~\ref{def:top-cotransverse} commute with positive scaling, as do the coface maps that insert a zero coordinate. By Lemma~\ref{lem:transverse-box-minus}, these maps give naturality on \(\A^-\). Taking the coend gives a natural homeomorphism \(B^-_\epsilon(K)\cong B^-_{\epsilon'}(K)\) for every \(\A\)-set \(K\).

\section{No colimit preservation for the naive definition}
\label{sec:old-branching}

In this section, \(\A\) is a thick category of cubes. For an \(\A\)-set \(K\), write \[K_n=K([n])\] for the set of \(n\)-cubes of \(K\). The \textit{initial vertex} of an \(n\)-cube \(c\) of an \(\A\)-set \(K\) is the vertex \[c^-=c(0_n).\] When \(c\in K_0\), one has \(c^-=c\). For \(\alpha\in K_0\), let \[\mathcal{C}^-_\alpha(K)=\{c\in K\mid \dim(c)\geq 1 \,\hbox{and}\, c^-=\alpha\}.\]

\begin{definition} [{\cite[Definition~4.2]{cubical-branching}}] \label{def:old-branching-space}
	Let \(\epsilon\in ]0,1[\). Let \(K\) be an \(\A\)-set. The \textit{naive \(\epsilon\)-branching space} \(\Pminus_\A(K,\epsilon)\) of \(K\) is the space 
	\[
	\Pminus_\A(K,\epsilon) = \coprod_{\alpha\in K_0} (\Pminus_\A)_\alpha(K,\epsilon)
	\]
	where 
	\[
	(\Pminus_\A)_\alpha(K,\epsilon) = \big\{|c|_{\geom}\circ\phi\mid c\in \mathcal{C}^-_\alpha(K)\hbox{ and }\phi\in N_{\dim(c)}(\epsilon)\big\}.
	\]
	Each summand is equipped with the \(\Delta\)-Kelleyfication of the subspace topology induced by the compact-open topology on paths \([0,\epsilon]\to |K|_{\geom}\). 
\end{definition}

\begin{proposition}
	Let \(\epsilon\in ]0,1[\). The assignment \(K\mapsto \Pminus_\A(K,\epsilon)\) induces a functor from \(\A\)-sets to topological spaces.
\end{proposition}

\begin{proof} Let \(f:K\to L\) be a map of \(\A\)-sets. For all \(c\in\mathcal{C}^-_\alpha(K)\), \(f(c)\in \mathcal{C}^-_{f(\alpha)}(L)\) and \(\dim(c)=\dim(f(c))\). The induced map is the restriction of postcomposition by
	\[
	|f|_{\mathrm{geom}}:|K|_{\mathrm{geom}}\longrightarrow |L|_{\mathrm{geom}}
	\]
	on compact-open path spaces. Since postcomposition by a continuous map is continuous for the compact-open topology, and since we then apply \(\Delta\)-Kelleyfication, the resulting map
	\[
	\Pminus_\A(K,\epsilon)\longrightarrow \Pminus_\A(L,\epsilon)
	\]
	is continuous. This proves the claim. 
\end{proof}

Let
\begin{align*}
	& c:[2]\longrightarrow [2],\qquad c(x_1,x_2)=(x_2,x_1),\\[1mm]
	& \gamma^+:[2]\longrightarrow [2],
	\qquad
	\gamma^+(x_1,x_2)=(\max(x_1,x_2),\min(x_1,x_2)),\\[1mm]
	& \gamma^-:[2]\longrightarrow [2],
	\qquad
	\gamma^-(x_1,x_2)=(\min(x_1,x_2),\max(x_1,x_2)).\\
\end{align*}
Recall that the \textit{rank} of a vertex
\(u=(u_1,\dots,u_n) \in [n]\) is
\[
r(u)=u_1+\cdots+u_n.
\]
For \(1\leq i\leq n\), let
\[
e_i=(0,\dots,0,1,0,\dots,0)
\]
be the vertex whose unique nonzero coordinate is in position \(i\). A \(\Sq\)-map
\(h:[2]\to[n]\) satisfying \(h(0,0)=0_n\) is necessarily of the form
\[
h(x,y)=
(0,\dots,0,x,0,\dots,0,y,0,
\dots,0)
\]
for a unique pair \(1\leq p<q\leq n\).  Hence
\[
h(1,0)=e_p,
\qquad
h(0,1)=e_q,
\qquad
p<q.
\]

\begin{lem} \label{lem:2face-restriction}
	Let
	\(
	f:[n]\longrightarrow [n]
	\)
	be a cotransverse map which is not the identity.  Then there exists a
	\(\Sq\)-map
	\[
	g:[2]\longrightarrow [n]
	\]
	such that the composite
	\[
	fg:[2]\longrightarrow [n]
	\]
	does not belong to the box category.
\end{lem}

\begin{proof}
	Since \(f\) is cotransverse, it is order-preserving and preserves the rank. Indeed, \(f(0_n)=0_n\) and \(f(1_n)=1_n\) by Proposition~\ref{prop:cotransverse-endo-endpoint-preserving}. If \(u\) has rank \(k\), choose a maximal chain from \(0_n\) to \(u\) of length \(k\) and a maximal chain from \(u\) to \(1_n\) of length \(n-k\). Since \(f\) preserves adjacency, these chains show respectively that \(r(f(u))\leq k\) and \(r(f(u))\geq k\). Hence \(r(f(u))=k\). In particular, it sends vertices of rank \(1\) to vertices of rank \(1\).  Therefore there is a unique map
	\[
	\alpha:\{1,\dots,n\}\longrightarrow \{1,\dots,n\}
	\]
	such that
	\[
	f(e_i)=e_{\alpha(i)}
	\qquad (1\leq i\leq n).
	\]
	We first prove that \(\alpha\neq\id_{\{1,\dots,n\}}\).  Indeed, suppose that \(\alpha=\id_{\{1,\dots,n\}}\).  Let \(u\in[n]\), and write
	\[
	\supp(u)=\{i\mid u_i=1\}.
	\]
	For every \(i\in\supp(u)\), one has \(e_i\leq u\).  Since \(f\) is order-preserving,
	\[
	e_i=f(e_i)\leq f(u).
	\]
	Thus
	\[
	\supp(u)\subseteq \supp(f(u)).
	\]
	Since \(f\) preserves the rank, we also have
	\[
	r(f(u))=r(u).
	\]
	Equivalently,
	\[
	\#\supp(f(u))=\#\supp(u).
	\]
	The inclusion of supports is therefore an equality.  Hence \(\supp(f(u))=\supp(u)\), and so \(f(u)=u\).  Since this holds for every vertex \(u\), we obtain \(f=\id_{[n]}\), contradicting the hypothesis. Consequently,
	\[
	\alpha\neq\id_{\{1,\dots,n\}}.
	\]
	Since \(\alpha\) is a non-identity endomap of the finite ordered set \(\{1,\dots,n\}\), there exist \(i<j\) such that
	\[
	\alpha(i)\geq \alpha(j).
	\]
	Indeed, if \(\alpha(i)<\alpha(j)\) for every \(i<j\), then \(\alpha\) would be strictly increasing.  A strictly increasing endomap of \(\{1,\dots,n\}\) is necessarily the identity. Choose such a pair \(i<j\), and let
	\[
	g=g_{ij}:[2]\longrightarrow[n]
	\]
	be the \(\Sq\)-map defined by putting the first variable in the \(i\)-th coordinate and the second variable in the \(j\)-th coordinate:
	\[
	g(x,y)=(0,\dots,0,x,0,\dots,0,y,0,\dots,0).
	\]
	Then
	\[
	fg(0,0)=f(0_n)=0_n,
	\]
	and
	\[
	fg(1,0)=f(e_i)=e_{\alpha(i)},
	\qquad
	fg(0,1)=f(e_j)=e_{\alpha(j)}.
	\]
	Suppose, for a contradiction, that
	\[
	fg\in\Sq([2],[n]).
	\]
	Since \(fg(0,0)=0_n\), the preceding description of \(\Sq\)-maps implies that there exist indices \(p<q\) such that
	\[
	fg(1,0)=e_p,
	\qquad
	fg(0,1)=e_q.
	\]
	Therefore
	\[
	p=\alpha(i),
	\qquad
	q=\alpha(j),
	\]
	and hence
	\[
	\alpha(i)<\alpha(j).
	\]
	This contradicts the choice of \(i<j\) with
	\(\alpha(i)\geq\alpha(j)\).  Thus
	\[
	fg\notin\Sq([2],[n]),
	\]
	as required.
\end{proof}

There exist thick categories of cubes containing \(\gamma^-\) and not \(\gamma^+\) or containing \(\gamma^+\) and not \(\gamma^-\): consider for the first case the category of cubes generated by the \(\Sq\)-maps and \(\gamma^-\) and for the second case the category of cubes generated by the \(\Sq\)-maps and \(\gamma^+\). Indeed, every morphism is either a \(\Sq\)-map or has standard cotransverse endomorphism part equal to the chosen generator. However, we have:

\begin{lem} \label{lem:thick-2dim-case}
	Assume that \(\Sq\subsetneq \A\). Then 
	\[
	\{c,\gamma^+,\gamma^-\} \cap \A([2],[2]) \neq \varnothing.
	\]
\end{lem}

\begin{proof}
	Let \(f:[m]\to [n]\in \A\backslash \Sq\). One always has 
	\[
	\A([0],[n]) = \Sq([0],[n]),\quad \A([1],[n]) = \Sq([1],[n])\quad\hbox{for all}\quad n\geq 0\,.
	\] 
	By the preceding observation, one must have \(m\geq 2\). If \(n=2\), then necessarily \(m=2\), and the conclusion follows directly. Assume that \(2\leq m<n\). Since \(\A\) is thick, \(f\) factors as a composite 
	\[
	f:[m]\xlongrightarrow{\phi\in\A} [m] \xlongrightarrow{\delta\in\Sq} [n]\,. 
	\]
	Since \(f\notin \Sq\), one deduces \(\phi\notin \Sq\). Thus we may assume without loss of generality that \(m=n\). By Lemma~\ref{lem:2face-restriction}, there exists a \(\Sq\)-map \(g:[2]\to [m]\) such that 
	\[
	fg\notin \Sq([2],[n])\,.
	\]
	Using thickness again, write \(fg\) as a composite 
	\[
	fg:[2]\xlongrightarrow{\phi\in\A} [2] \xlongrightarrow{\delta\in\Sq} [m]\,. 
	\]
	Then \(\phi\notin\Sq\). Thus \[\phi\in  \{c,\gamma^+,\gamma^-\}\,.\]
	This completes the proof.
\end{proof}

\begin{lem} \label{lem:c-case}
	Let \(0<\epsilon<1\). Let \(\A\) be a category of cubes containing the cotransverse endomorphism
	\[
	c:[2]\longrightarrow [2].
	\] 
	Consider the coequalizer
	\[
	\begin{tikzcd}[column sep=3em]
		{\A[2]}
		\arrow[r, shift left=1ex, "\id"]
		\arrow[r, shift right=1ex, "{\A(-,c)}"']
		&
		{\A[2]}
		\arrow[r, "q"]
		&
		{Q_c}
	\end{tikzcd}
	\]
	in \(\A^{\op}\Set\). Then there exist two natural directed paths
	\[
	u,v:[0,\epsilon]\longrightarrow [0,1]^2
	\]
	of natural length \(\epsilon\) such that:
	\begin{enumerate}
		\item \(qu=qv\) in \(\Pminus_\A(Q_c,\epsilon)\);
		\item \(u\) and \(v\) are not equivalent for the equivalence relation on
		\(\Pminus_\A(\A[2],\epsilon)\) generated by \(w\sim cw\).
	\end{enumerate}
\end{lem}

\begin{proof}
	Put \(a=\epsilon/3\). Define
	\[
	u(t)=
	\begin{cases}
		(0,t), & 0\leq t\leq a,\\
		(t-a,a), & a\leq t\leq 2a,\\
		(a,t-a), & 2a\leq t\leq 3a,
	\end{cases}
	\]
	and
	\[
	v(t)=
	\begin{cases}
		(0,t), & 0\leq t\leq a,\\
		(t-a,a), & a\leq t\leq 3a.
	\end{cases}
	\]
	Both paths start from \((0,0)\), are coordinatewise nondecreasing, and satisfy
	\[
	t=u_1(t)+u_2(t)=v_1(t)+v_2(t)
	\]
	for every \(0\leq t\le\epsilon\). Hence they are natural directed paths of natural length \(\epsilon\). We prove that \(qu=qv\). For \(0\leq t\leq 2a\), one has \(u(t)=v(t)\). For
	\(2a\leq t\leq 3a\), one has
	\[
	c(u(t))=c(a,t-a)=(t-a,a)=v(t).
	\]
	Since the quotient \(Q_c\) identifies every point \(z\in [0,1]^2\) with \(c(z)\), it follows that \(q(u(t))=q(v(t))\) for every \(t\). Therefore
	\[
	qu=qv
	\]
	as elements of \(\Pminus_\A(Q_c,\epsilon)\). It remains to show that \(u\) and \(v\) are not equivalent in the coequalizer of path spaces. Since \(c^2=\mathrm{id}\), the equivalence relation generated by \(w\sim cw\) has equivalence classes equal to the orbits of the involution \(w\mapsto cw\). Thus the equivalence class of \(u\) is exactly
	\[
	\{u,cu\}.
	\]
	The path \(v\) is not equal to \(u\), since they differ on the interval \([2a,3a]\). Also \(v\ne cu\), since on the interval \([0,a]\) one has
	\[
	v(t)=(0,t)
	\qquad\text{whereas}\qquad
	cu(t)=(t,0).
	\]
	Hence \(v\notin\{u,cu\}\). For the uniform metric on paths, distance to \(\{u,cu\}\) is continuous and \(c\)-invariant. It therefore separates the two classes even in the \(\Delta\)-Hausdorff coequalizer.
\end{proof}

\begin{lem}\label{lem:case-product-closed}
Let \(\mathcal A\) be a category of cubes containing the two cotransverse
			endomorphisms
			\[
			r=\gamma^-_{1,2},\qquad s=\gamma^-_{2,3}
			\]
			of \([3]\), where
			\[
			r(x,y,z)=(\min(x,y),\max(x,y),z),
			\]
			and
			\[
			s(x,y,z)=(x,\min(y,z),\max(y,z)).
			\]
			Let \(0<\epsilon<1\). Then the naive \(\epsilon\)-branching-space functor
			\[
			\Pminus_\A(-,\epsilon):\A^{\op}\Set
			\longrightarrow \Top
			\]
			does not preserve the coequalizer
			\[
			\begin{tikzcd}[column sep=5em]
				{\A[3]}
				\arrow[r, shift left=1ex, "{\A(-,r)}"]
				\arrow[r, shift right=1ex, "{\A(-,s)}"']
				&
				{\A[3]}
				\arrow[r, "q"]
				&
				{Q}.
			\end{tikzcd}
			\]
			In particular, \(\Pminus_\A(-,\epsilon)\) is not colimit-preserving.
\end{lem}

\begin{proof}
		In the geometric realization of \(Q\), one has
		\[
		q\,\T(r)(p)=q\,\T(s)(p)
		\]
		for every \(p\in [0,1]^3\). Put
		\[
		a=\varepsilon/4.
		\]
		Define a natural directed path
		\[
		w:[0,\varepsilon]\longrightarrow [0,1]^3
		\]
		by
		\[
		w(t)=
		\begin{cases}
			(t,0,0), & 0\leq t\leq a,\\
			(a,t-a,0), & a\leq t\leq 2a,\\
			(a,a,t-2a), & 2a\leq t\leq 3a,\\
			(t-2a,a,a), & 3a\leq t\leq 4a.
		\end{cases}
		\]
		It is coordinatewise nondecreasing and satisfies
		\[
		t=w_1(t)+w_2(t)+w_3(t)
		\]
		for every \(0\leq t\le\varepsilon\). Hence \(w\in N_3(\varepsilon)\). Set
		\[
		u=\T(r)w.
		\]
		Explicitly,
		\[
		u(t)=
		\begin{cases}
			(0,t,0), & 0\leq t\leq a,\\
			(t-a,a,0), & a\leq t\leq 2a,\\
			(a,a,t-2a), & 2a\leq t\leq 3a,\\
			(a,t-2a,a), & 3a\leq t\leq 4a.
		\end{cases}
		\]
		Now define another path \(v:[0,\varepsilon]\to [0,1]^3\) by
		\[
		v(t)=
		\begin{cases}
			u(t), & 0\leq t\leq 3a,\\
			\T(s)w(t), & 3a\leq t\leq 4a.
		\end{cases}
		\]
		Thus
		\[
		v(t)=
		\begin{cases}
			(0,t,0), & 0\leq t\leq a,\\
			(t-a,a,0), & a\leq t\leq 2a,\\
			(a,a,t-2a), & 2a\leq t\leq 3a,\\
			(t-2a,a,a), & 3a\leq t\leq 4a.
		\end{cases}
		\]
		The two pieces agree at \(t=3a\), since
		\[
		\T(r)w(3a)=\T(s)w(3a)=(a,a,a).
		\]
		Moreover \(v\) is coordinatewise nondecreasing and satisfies
		\[
		t=v_1(t)+v_2(t)+v_3(t)
		\]
		for all \(t\). Hence \(v\in N_3(\varepsilon)\). We claim that \(qu=qv\) in \(\Pminus_\A(Q,\varepsilon)\). For \(0\leq t\leq 3a\), one has \(u(t)=v(t)\). For \(3a\leq t\leq 4a\), one has
		\[
		u(t)=\T(r)w(t),\qquad v(t)=\T(s)w(t).
		\]
		Since \(Q\) coequalizes \(r\) and \(s\), it follows that
		\[
		q(u(t))=q(v(t))
		\]
		for every \(t\). Therefore
		\[
		qu=qv.
		\]
		It remains to prove that \(u\) and \(v\) are not identified in the coequalizer of path spaces. Let
		\[
		L_{12}=\{(x,y,z)\in [0,1]^3\mid x\leq y\},
		\]
		and
		\[
		L_{23}=\{(x,y,z)\in [0,1]^3\mid y\leq z\}.
		\]
		For every \(\omega\in \Pminus_\A(\A[3],\epsilon)\), the path \(\T(r)\omega\) has image in \(L_{12}\), and the path \(\T(s)\omega\) has image in \(L_{23}\). Consequently every elementary identification in the coequalizer of path spaces has one endpoint whose image is contained in \(L_{12}\) and the other endpoint whose image is contained in \(L_{23}\). The path \(v\) is not contained in \(L_{12}\). Indeed, for \(3a<t<4a\),
		\[
		v(t)=(t-2a,a,a),
		\]
		and therefore
		\[
		v_1(t)=t-2a>a=v_2(t).
		\]
		The path \(v\) is not contained in \(L_{23}\) either. Indeed, for \(0<t<a\),
		\[
		v(t)=(0,t,0),
		\]
		and therefore
		\[
		v_2(t)=t>0=v_3(t).
		\]
		Thus \(v\) is neither of the form \(\T(r)\omega\) nor of the form \(\T(s)\omega\) for any path \(\omega\). Hence no elementary identification in the coequalizer of path spaces has \(v\) as one of its endpoints. Therefore the equivalence class of \(v\) is the singleton \(\{v\}\). On the other hand, \(u\ne v\), since for \(3a<t<4a\),
		\[
		u(t)=(a,t-2a,a),
		\qquad
		v(t)=(t-2a,a,a).
		\]
		The paths contained in \(L_{12}\) or contained in \(L_{23}\) form a closed set in the uniform metric. Distance to this set vanishes on the images of both maps and is positive at \(v\), so it also separates \(u\) and \(v\) in the \(\Delta\)-Hausdorff coequalizer. Hence \(u\) and \(v\) represent two distinct elements of
		\[
		\operatorname{coeq}\Big(
		\Pminus_\A(\mathcal A[3],\epsilon)
		\rightrightarrows
		\Pminus_\A(\mathcal A[3],\epsilon)
		\Big),
		\]
		but they have the same image in \(\Pminus_\A(Q,\epsilon)\). The canonical comparison map is therefore not injective. Thus \(\Pminus_\A(-,\epsilon)\) does not preserve this coequalizer.
\end{proof}

\begin{definition}
	A category of cubes \(\A\) is \textit{product-closed} if for any morphisms 
	\[
	f:[m]\longrightarrow[n],\qquad f':[m']\longrightarrow [n']
	\]
	of \(\A\), the product map 
	\[
	(f,f'):[m+m'] \longrightarrow [n+n']
	\]
	belongs to \(\A\).
\end{definition}

\begin{thm} \label{thm:no-colimit-preserving}
	Let \(\A\) be a product-closed thick category of cubes. The following three conditions are equivalent: 
	\begin{enumerate}
		\item \(\A=\Sq\),
		\item For every \(0<\epsilon<1\), the naive \(\epsilon\)-branching-space functor
		\[
		\Pminus_\A(-,\epsilon):\A^{\op}\Set\longrightarrow\Top
		\]
		is colimit-preserving.
		\item For some \(0<\epsilon<1\), the functor \(\Pminus_\A(-,\epsilon)\) is colimit-preserving.
	\end{enumerate}
\end{thm}

\begin{proof}
	The direction \((1)\Rightarrow (2)\) is \cite[Theorem~5.5]{cubical-branching}. Taking \(\epsilon=\tfrac12\) gives \((2)\Rightarrow(3)\). Assume now that \(\A\neq\Sq\) and fix any \(0<\epsilon<1\). Using Lemma~\ref{lem:thick-2dim-case}, choose \(\theta\in\{c,\gamma^+,\gamma^-\}\cap\A([2],[2])\). If \(\theta=c\), then the naive \(\epsilon\)-branching-space functor is not colimit-preserving by Lemma~\ref{lem:c-case}. If \(\theta=\gamma^-\), then product-closure gives
	\[
	\gamma^-_{1,2}=(\gamma^-,\mathrm{id}_{[1]}),
	\qquad
	\gamma^-_{2,3}=(\mathrm{id}_{[1]},\gamma^-)
	\]
	as morphisms of \(\mathcal A\). Hence Lemma~\ref{lem:case-product-closed} applies. If \(\theta=\gamma^+\), use an analogous argument as in Lemma~\ref{lem:case-product-closed} with
	\[
	r=\gamma^+_{1,2}=(\gamma^+,\id),\qquad s=\gamma^+_{2,3}=(\id,\gamma^+).
	\]
	Thus \(\Pminus_\A(-,\epsilon)\) fails to preserve a coequalizer for every \(0<\epsilon<1\). This proves \((3)\Rightarrow(1)\).
\end{proof}

In the case where \(\A\) is not product-closed, we can also find a counterexample if both \(\gamma^-\) and \(\gamma^+\) are in \(\A\) as follows: 

\begin{proposition} \label{prop:gamma-case}
	Let \(\A\) be a category of cubes containing the two cotransverse endomorphisms
	\[
	\gamma^-,\gamma^+:[2]\longrightarrow [2].
	\]
	Let
	\[
	\begin{tikzcd}[column sep=5em]
		{\A[2]}
		\arrow[r, shift left=1ex, "{\A(-,\gamma^-)}"]
		\arrow[r, shift right=1ex, "{\A(-,\gamma^+)}"']
		&
		{\A[2]}
		\arrow[r, "q"]
		&
		{Q}
	\end{tikzcd}
	\]
	be the coequalizer in \(\A^{\op}\Set\). Then, for every
	\(0<\epsilon<1\), the canonical map
	\[
	\operatorname{coeq}\Big(
	\Pminus_\A(\A[2],\epsilon)
	\rightrightarrows
	\Pminus_\A(\A[2],\epsilon)
	\Big)
	\longrightarrow
	\Pminus_\A(Q,\epsilon)
	\]
	is not injective. Consequently \(\Pminus_\A(-,\epsilon)\) does not preserve this
	coequalizer.
\end{proposition}

\begin{proof}
	Write
	\[
	r=\T(\gamma^-),\qquad s=\T(\gamma^+).
	\]
	Thus
	\[
	r(x,y)=(\min(x,y),\max(x,y)),
	\qquad
	s(x,y)=(\max(x,y),\min(x,y)).
	\]
	Let
	\[
	c(x,y)=(y,x).
	\]
	Then
	\[
	s=cr.
	\]
	Put \(a=\epsilon/3\). Define two paths
	\[
	u,v:[0,\epsilon]\longrightarrow [0,1]^2
	\]
	by
	\[
	u(t)=
	\begin{cases}
		(0,t), & 0\leq t\leq a,\\
		(t-a,a), & a\leq t\leq 2a,\\
		(a,t-a), & 2a\leq t\leq 3a,
	\end{cases}
	\]
	and
	\[
	v(t)=
	\begin{cases}
		(0,t), & 0\leq t\leq a,\\
		(t-a,a), & a\leq t\leq 3a.
	\end{cases}
	\]
	Both paths start from \((0,0)\), are coordinatewise nondecreasing, and satisfy
	\[
	t=u_1(t)+u_2(t)=v_1(t)+v_2(t)
	\]
	for every \(0\leq t\leq \epsilon\). Hence \(u,v\in N_2(\epsilon)\). We first prove that \(qu=qv\) in \(\Pminus_\A(Q,\epsilon)\). For \(0\leq t\leq 2a\),
	one has
	\[
	u(t)=v(t).
	\]
	For \(2a\leq t\leq 3a\), one has
	\[
	u(t)=(a,t-a),
	\qquad
	v(t)=(t-a,a).
	\]
	Since \(a\leq t-a\), this gives
	\[
	r(v(t))=(a,t-a)=u(t),
	\qquad
	s(v(t))=(t-a,a)=v(t).
	\]
	Thus \(u(t)\) and \(v(t)\) have the same image in the quotient \(Q\) for every
	\(t\). Therefore
	\[
	qu=qv
	\]
	as elements of \(\Pminus_\A(Q,\epsilon)\). It remains to show that \(u\) and \(v\) are not identified in the coequalizer of
	path spaces. Let
	\[
	L=\{(x,y)\in [0,1]^2\mid x\leq y\},
	\qquad
	U=\{(x,y)\in [0,1]^2\mid x\geq y\}.
	\]
	For every \(w\in \Pminus_\A(\A[2],\epsilon)\), the path \(rw\) has image in
	\(L\), and the path \(sw\) has image in \(U\). Moreover
	\[
	sw=crw.
	\]
	Hence every elementary identification in the coequalizer of path spaces is of the form
	\[
	p\sim cp
	\]
	where \(p=rw\) for some \(w\in \Pminus_\A(\A[2],\epsilon)\); in particular
	\(p([0,\epsilon])\subseteq L\).
	The path \(u\) has image in \(L\). Therefore the only elementary identification involving \(u\) is
	\[
	u\sim cu.
	\]
	Indeed, if \(u=rw\), then
	\[
	sw=crw=cu,
	\]
	and if \(u=sw\), then \(u\) would have image in \(U\), which is false because
	\(u(t)\in L\) for all \(t\) and \(u(t)\notin U\) for \(0<t<a\). Thus the equivalence class of \(u\) is contained in
	\[
	\{u,cu\}.
	\]
	Now \(v\ne u\), since on \(2a<t<3a\) the two paths differ. Also \(v\ne cu\), since on \(0<t<a\) one has
	\[
	v(t)=(0,t),
	\qquad
	cu(t)=(t,0).
	\]
	As above, distance to \(\{u,cu\}\) is \(c\)-invariant, so it coequalizes \(r\) and \(s=cr\) and separates \(u\) from \(v\). Therefore the two distinct classes represented by \(u\) and \(v\) in
	\[
	\operatorname{coeq}\Big(
	\Pminus_\A(\A[2],\epsilon)
	\rightrightarrows
	\Pminus_\A(\A[2],\epsilon)
	\Big)
	\]
	have the same image in \(\Pminus_\A(Q,\epsilon)\). The canonical comparison map is not injective. Consequently \(\Pminus_\A(-,\epsilon)\) does not preserve the above coequalizer.
\end{proof}

\section{Branching spaces of free transverse sets}
\label{sec:free-transverse}

We return to the \(\A\)-sets arising as concrete models of concurrent processes, namely the free extensions \(\LA(K)\) of precubical sets \(K\). We compare the branching space of \(\LA(K)\) with the geometric branching space of \(K\). Let
\[
\rho_\square:\square^{\op}\Set\longrightarrow(\square^-)^{\op}\Set
\]
denote the restriction functor along the inclusion \(\square^-\subseteq\square\).

\begin{lem}\label{lem:minus-base-change-representables}
	Let \(\A\) be a thick category of cubes. For every \(n\geq 0\), the canonical map of
	\((\A^-)^{\op}\Set\)
	\[
	\LA^-\rho_{\Sq}(\Sq[n])
	\longrightarrow
	\rho_\A\LA(\Sq[n])
	\]
	is an isomorphism.
\end{lem}

\begin{proof}
	Since \(\LA(\Sq[n])\cong \A(-,[n])\), it suffices to prove that, for every \(r\geq 0\), the canonical map
	\[
	\Phi_{r,n}:
	\int^{[q]\in\Sq^-}
	\Sq([q],[n])\times \A^-([r],[q])
	\longrightarrow
	\A([r],[n])
	\]
	induced by composition is a bijection. Explicitly,
	\[
	\Phi_{r,n}([\alpha,g])=\alpha g
	\]
	for \(g:[r]\to[q]\) in \(\A^-\) and \(\alpha:[q]\to[n]\) in \(\Sq\). We first prove surjectivity. Let \(f:[r]\to[n]\) be a morphism of \(\A\). Write its cotransverse--cocubical factorization as
	\[
	f=\delta\psi,
	\qquad
	\psi:[r]\to[r],
	\qquad
	\delta:[r]\to[n],
	\]
	with \(\delta\in\Sq([r],[n])\). Since \(\A\) is thick, the cotransverse endomorphism \(\psi\) belongs to \(\A([r],[r])\). By Proposition~\ref{prop:cotransverse-endo-endpoint-preserving}, it preserves the initial vertex, and therefore
	\[
	\psi\in\A^-([r],[r]).
	\]
	Thus \(f\) is the image of the class represented by
	\[
	[r]\xrightarrow{\psi}[r]\xrightarrow{\delta}[n].
	\]
	We now prove injectivity. Suppose that two representatives
	\[
	[r]\xrightarrow{g}[q]\xrightarrow{\alpha}[n],
	\qquad
	[r]\xrightarrow{h}[q']\xrightarrow{\beta}[n]
	\]
	have the same image in \(\A([r],[n])\), so that
	\[
	\alpha g=\beta h.
	\]
	Since \(g\) and \(h\) are morphisms of \(\A^-\), Lemma~\ref{lem:transverse-box-minus} gives their cotransverse--cocubical factorizations inside \(\A^-\):
	\[
	g=\epsilon\psi,
	\qquad
	h=\eta\chi,
	\]
	with
	\[
	\epsilon\in\Sq^-([r],[q]),
	\qquad
	\eta\in\Sq^-([r],[q']),
	\qquad
	\psi,\chi\in\A^-([r],[r]).
	\]
	The defining relations of the coend give
	\[
	[\alpha,g]=[\alpha\epsilon,\psi],
	\qquad
	[\beta,h]=[\beta\eta,\chi].
	\]
	The equality \(\alpha g=\beta h\) becomes
	\[
	\alpha\epsilon\psi=\beta\eta\chi.
	\]
	Here \(\alpha\epsilon\) and \(\beta\eta\) are morphisms of \(\Sq([r],[n])\). By the uniqueness of the cotransverse--cocubical factorization, we obtain
	\[
	\alpha\epsilon=\beta\eta
	\qquad\text{and}\qquad
	\psi=\chi.
	\]
	Therefore
	\[
	[\alpha,g]
	=[\alpha\epsilon,\psi]
	=[\beta\eta,\chi]
	=[\beta,h].
	\]
	This proves injectivity. Hence \(\Phi_{r,n}\) is a bijection for all \(r,n\), and the canonical map is an isomorphism of \((\A^-)^{\op}\Set\).
\end{proof}

\begin{lem}\label{lem:minus-base-change}
	Let \(\A\) be a thick category of cubes. There is a natural isomorphism of functors
	\[
	\LA^-\rho_{\Sq}
	\cong
	\rho_\A\LA
	:
	\Sq^{\op}\Set\longrightarrow (\A^-)^{\op}\Set .
	\]
\end{lem}

\begin{proof}
	Both functors preserve colimits. Indeed, \(\LA\) and \(\LA^-\) are left adjoints, and \(\rho_\A\) and \(\rho_{\Sq}\) preserve colimits because colimits of presheaves are computed pointwise. Since every precubical set is a colimit of representables, it is enough to compare the two functors on the representables \(\Sq[n]\). This is precisely Lemma~\ref{lem:minus-base-change-representables}.
\end{proof}

\begin{lem}\label{lem:Nhat-after-Lminus}
	For every \(\Sq^-\)-set \(H\), there is a natural homeomorphism
	\[
	\widehat{N_\epsilon}(\LA^- H)
	\cong
	\int^{[q]\in\Sq^-} H([q])\cdot N_q(\epsilon).
	\]
\end{lem}

\begin{proof}
	Using the coend formula for the left Kan extension, Fubini's theorem for coends, and the co-Yoneda lemma, we obtain the following natural sequence of homeomorphisms:
	\begin{align*}
		\widehat{N_\epsilon}(\LA^- H)
		&=
		\int^{[p]\in\A^-}
		(\LA^- H)([p])\cdot N_p(\epsilon)
		\nobreak\\
		&\cong
		\int^{[p]\in\A^-}
		\left(
		\int^{[q]\in\Sq^-}
		H([q])\times \A^-([p],[q])
		\right)
		\cdot N_p(\epsilon)
		\nobreak\\
		&\cong
		\int^{[q]\in\Sq^-}
		H([q])\cdot
		\left(
		\int^{[p]\in\A^-}
		\A^-([p],[q])\cdot N_p(\epsilon)
		\right)
		\nobreak\\
		&\cong
		\int^{[q]\in\Sq^-}
		H([q])\cdot N_q(\epsilon).
	\end{align*}
\end{proof}

For a vertex \(a=(a_1,\ldots,a_n)\in [n]\), put \[ d_a=n-(a_1+\cdots+a_n)=\dd(a,1_n). \] Let \[ \kappa_a:[d_a]\longrightarrow [n] \] be the unique box map whose image is the face spanned by the interval \([a,1_n]\), that is, the unique box map satisfying \[ \kappa_a(0_{d_a})=a \qquad\text{and}\qquad \kappa_a(1_{d_a})=1_n. \] Explicitly, \(\kappa_a\) fixes to \(1\) the coordinates \(i\) such that \(a_i=1\), and inserts the \(d_a\) variables, in their natural order, in the coordinates \(i\) such that \(a_i=0\). We call the representable \(\Sq^-(-,[r])\) the \textit{standard \(r\)-corner}. With this terminology, the restriction of the representable precubical \(n\)-cube to \(\Sq^-\) is the coproduct of its corners: \[ \rho_\Sq(\Sq[n]) \cong \displaystyle\coprod_{a\in [n]} \Sq^-(-,[d_a]). \] Indeed, at \([q]\) the inverse isomorphism sends a morphism \[ \beta:[q]\longrightarrow [d_a] \] of \(\Sq^-\) to the composite \[ [q]\xrightarrow{\beta}[d_a]\xrightarrow{\kappa_a}[n]. \] This map is bijective. If \(\alpha:[q]\to[n]\) is a box map, then its initial vertex \(a=\alpha(0_q)\) determines the component. For every coordinate \(i\) with \(a_i=1\), the \(i\)-th coordinate of \(\alpha\) is constantly \(1\). On the remaining coordinates, \(\alpha\) is uniquely encoded by a morphism \[ \beta:[q]\longrightarrow [d_a] \] of \(\Sq^-\), and \(\alpha=\kappa_a\beta\). This decomposition is clearly natural with respect to precomposition by morphisms of \(\Sq^-\).


\begin{lem}\label{lem:old-precubical-coend}
	For every precubical set \(K\), there is a natural homeomorphism
	\[
	\int^{[q]\in\Sq^-} K_q\cdot N_q(\epsilon)
	\xlongrightarrow{\cong}
	\Pminus_\Sq(K,\epsilon).
	\]
	For \(K=\Sq[n]\), this homeomorphism sends the class of a pair
	\[
	\alpha:[q]\to[n],
	\qquad
	u\in N_q(\epsilon),
	\]
	with \(\alpha\in\Sq([q],[n])\), to the short natural directed path
	\[
	\T(\alpha)u\in \Pminus_\Sq(\Sq[n],\epsilon).
	\]
\end{lem}

\begin{proof} 
	The displayed map is well-defined because, for every morphism \(v:[p]\to[q]\) of \(\Sq^-\), one has \[ \T(\alpha v)=\T(\alpha)\T(v), \] and the coend relation identifies \((\alpha v,w)\) with \((\alpha,\T(v)w)\) for \(w\in N_p(\epsilon)\). Both sides are colimit-preserving with respect to the precubical set \(K\): the left-hand side is colimit-preserving by the coend formula, and the right-hand side is colimit-preserving by \cite[Theorem~5.5]{cubical-branching}. Since every precubical set is a colimit of representables, it remains to prove the result for \(K=\Sq[n]\). Using the corner decomposition above, we obtain \[ \begin{aligned} \int^{[q]\in\Sq^-}\Sq([q],[n])\cdot N_q(\epsilon) &\cong \int^{[q]\in\Sq^-} \left( \displaystyle\coprod_{a\in[n]}\Sq^-([q],[d_a]) \right)\cdot N_q(\epsilon) \\ &\cong \displaystyle\coprod_{a\in[n]} \int^{[q]\in\Sq^-} \Sq^-([q],[d_a])\cdot N_q(\epsilon) \\ &\cong \displaystyle\coprod_{a\in[n]} N_{d_a}(\epsilon), \end{aligned} \] by the co-Yoneda lemma. Under this identification, the summand indexed by \(a\) is sent to \(\Pminus_\Sq(\Sq[n],\epsilon)\) by \[ u\longmapsto \T(\kappa_a)u. \] Its image is exactly the space of short natural directed paths in the precubical cube \(\Sq[n]\) starting from the vertex \(a\): the map \(\kappa_a\) identifies the standard \(d_a\)-corner with the face interval \([a,1_n]\). Taking the coproduct over all vertices \(a\in[n]\) gives precisely \(\Pminus_\Sq(\Sq[n],\epsilon)\). This proves the assertion for representables, and hence for all precubical sets by colimit preservation. 
\end{proof}

\begin{thm}\label{thm:comparison-with-precubical-branching}
	Let \(\A\) be a thick category of cubes, and let \(\epsilon\in]0,1[\). For every precubical set \(K\), there is a natural homeomorphism
	\[
	\widehat{N_\epsilon}\bigl(\rho_\A\LA(K)\bigr)
	\cong
	\Pminus_\Sq(K,\epsilon).
	\]
	Equivalently,
	\[
	B^-_\epsilon(\LA(K))
	\cong
	\Pminus_\Sq(K,\epsilon),
	\]
	where \(B^-_\epsilon=\widehat{N_\epsilon}\rho_\A\).
\end{thm}

\begin{proof}
	By Lemma~\ref{lem:minus-base-change}, there is a natural isomorphism
	\[
	\rho_\A\LA(K)
	\cong
	\LA^-\rho_{\Sq}(K).
	\]
	Applying \(\widehat{N_\epsilon}\) and then Lemma~\ref{lem:Nhat-after-Lminus}, we obtain
	\[
	\begin{aligned}
		\widehat{N_\epsilon}\bigl(\rho_\A\LA(K)\bigr)
		&\cong
		\widehat{N_\epsilon}\bigl(\LA^-\rho_{\Sq}(K)\bigr) \\
		&\cong
		\int^{[q]\in\Sq^-} \rho_{\Sq}(K)([q])\cdot N_q(\epsilon) \\
		&=
		\int^{[q]\in\Sq^-} K_q\cdot N_q(\epsilon).
	\end{aligned}
	\]
	The last coend is naturally homeomorphic to \(\Pminus_\Sq(K,\epsilon)\) by Lemma~\ref{lem:old-precubical-coend}. This proves the assertion. 
\end{proof}

For the representable precubical set \(K=\Sq[n]\), the resulting homeomorphism is the explicit map
\[
\widehat{N_\epsilon}\bigl(\rho_\A\LA(\Sq[n])\bigr)
\cong
\int^{[q]\in\Sq^-}\Sq([q],[n])\cdot N_q(\epsilon)
\longrightarrow
\Pminus_\Sq(\Sq[n],\epsilon)
\]
sending \([\alpha,u]\) to \(\T(\alpha)u\). Notice that \(\alpha\) is an arbitrary morphism of \(\Sq\), not necessarily a morphism of \(\Sq^-\). This is why the construction recovers short natural directed paths starting from every vertex of the precubical cube \(\Sq[n]\), not only those starting from \(0_n\).


\end{document}